\documentclass[12pt]{amsart}
\usepackage{amsmath,times,epsfig,amssymb,amsbsy,amscd,amsfonts,amstext,color,bm}
\usepackage{mathtools}
\usepackage[margin=1in]{geometry}
\usepackage{enumerate}
\usepackage{placeins}
\usepackage{array}
\usepackage{tabulary}
\usepackage[justification=centering]{caption}
\usepackage{multirow}
\usepackage{graphicx}
\usepackage{mathrsfs}
\usepackage{hyperref}
\usepackage{hhline}
\usepackage{multirow}
\usepackage[table, svgnames, dvipsnames]{xcolor}
\usepackage{makecell, cellspace, caption}
\usepackage{subcaption}
\usepackage[color,matrix,arrow]{xypic}
\usepackage{tikz}
        \usetikzlibrary{arrows.meta}
        \usetikzlibrary{arrows}
        \usetikzlibrary{quotes}
        \usetikzlibrary{graphs}
        \usetikzlibrary{positioning}
\usepackage{tikz-cd}
\hypersetup{
  colorlinks   = true, 
  urlcolor     = blue, 
  linkcolor    = blue, 
  citecolor   = red 
}

\theoremstyle{plain}
\newtheorem{theorem}{Theorem}[section]
\newtheorem{corollary}[theorem]{Corollary}
\newtheorem{lemma}[theorem]{Lemma}

\newtheorem{proposition}[theorem]{Proposition}

    

\theoremstyle{definition}
\newtheorem{definition}[theorem]{Definition}
\newtheorem{example}[theorem]{Example}
\newtheorem{notation}[theorem]{Notation}

\theoremstyle{remark}
\newtheorem{remark}[theorem]{Remark}

\newcommand{\A}{\mathcal{A}}
\newcommand{\scB}{\mathcal{B}}
\newcommand{\scV}{\mathcal{V}}
\newcommand{\scE}{\mathcal{E}}
\newcommand{\C}{\mathbb{C}}

\newcommand{\scL }{\mathcal{L}}
\newcommand{\scN}{\mathcal{N}}

\newcommand{\Q}{\mathbb{Q}}
\newcommand{\R}{\mathbb{R}}
\newcommand{\K}{\mathbb{K}}

\newcommand{\Z}{\mathbb{Z}}

\newcommand{\M}{\mathcal{M}}

\newcommand{\codim}{\operatorname{codim}}

\newcommand{\Hom}{\operatorname{Hom}}

\newcommand{\tor}{\operatorname{tor}}

\newcommand{\rank}{\operatorname{rank}}

\newcolumntype{K}[1]{>{\centering\arraybackslash}p{#1}}

\allowdisplaybreaks
\begin{document}

\title[Expectations of Tutte-related functions]{Expectations of Tutte-related functions of random ranked sets with multiplicities}

\date{\today}

\begin{abstract}
\label{sec:intro}
Employing two models, we show that various counting functions of a random variable defined by restriction or contraction of a ranked set with multiplicity (e.g., classical and arithmetic matroids) have expectations given by the corresponding multivariate Tutte polynomial. The first model is based on a generalization of a convolution formula of Kung (2010), extending from matroids to ranked sets with multiplicities. This model enables us to compute the expectations of many familiar polynomials, such as the chromatic, flow and Ehrhart polynomials, generalizing the classical results of Welsh (1996) on random graphs. The second model is designed to compute the expectations of invariants that are generally not evaluations of the polynomials mentioned above, such as the number of connected components of an intersection of hypersurfaces in an abelian Lie group arrangement, and the number of lattice points in a half-open zonotope. In particular, both models yield new probabilistic interpretations of the arithmetic Tutte polynomial and $G$-Tutte polynomial. A simple, but seems to be new convolution-like formula for the Ehrhart polynomials of lattice zonotopes will also be given.
 \end{abstract}

\author{Tan Nhat Tran}
\address{Tan Nhat Tran, Department of Mathematics, Hokkaido University, Kita 10, Nishi 8, Kita-Ku, Sapporo 060-0810, Japan.}
\email{trannhattan@math.sci.hokudai.ac.jp}



\date{\today}
\maketitle


\section{Introduction}

\subsection{Setup}
\label{subsec:setup} 
A \emph{ranked set with multiplicity} is a triple $\M=(E,r,m)$, where $E$ is a finite set, $r : 2^E \to \Z$ is a rank function, and $m : 2^E \to R$ is a multiplicity function, where $R$ denotes a commutative ring with $1$. 
Here $r(\emptyset)$ is possibly non-zero, and  $m$ is possibly non-trivial.
Denote $\underline{v}=\{v_e: e\in E\}$, and for $A\subseteq E$, define  $\underline{v}^A :=\prod_{e \in A}v_e$.
Let $\M$ be a ranked set with multiplicity. We associate to $\M$ the (Laurent) polynomial 
$$
\mathbf{Z}_{\M} (q, \underline{v})=
\sum_{A\subseteq E}m(A)q^{-r(A)}\underline{v}^A \in R[q^{\pm1}, \underline{v}]. 
$$
Main examples are:
\begin{itemize}
\item $\M$ is a \emph{matroid}, where $r$ satisfies the matroid rank axioms (in particular,  $r(\emptyset)=0$), and $m$ is trivial (e.g., see \cite{O92} and also Example \ref{eg:matroid}).
In this case, $\mathbf{Z}_{\M} (q, \underline{v})$ is equal to the \emph{multivariate Tutte polynomial} \cite{sokal}, which  via the change of variables $q=(x-1)(y-1)$, $v_e=y-1$ for all $e\in E$ specializes to the classical \emph{Tutte polynomial} of the matroid
$$T _{\M} (x, y) 
=\sum_{ A\subseteq E} (x-1)^{r(E)-r(A)}(y-1)^{|A|-r(A)}.$$
\item $\M$ is an \emph{arithmetic matroid}, where $(E,r)$ is a matroid, and $m$ satisfies certain divisibility and positivity axioms (see \cite{DM13, BM14} and also Example \ref{eg:arith-matroid}). 
In this case, $\mathbf{Z}_{\M} (q, \underline{v})$ is equal to the \emph{multivariate arithmetic Tutte polynomial} \cite{BM14}. 
\item $\M$ is a \emph{$\Z$-representable matroid with $G$-multiplicity}, where $(E,r)$ is a representable matroid over $\Z$ in the sense of Fink-Moci \cite{FM16}, and $m$ is defined by means of the number of homomorphisms from a finite abelian group to a given abelian group $G$ (see \cite{LTY} and also Example \ref{eg:g-mult}). 
In this case, $\mathbf{Z}_{\M} (q, \underline{v})$ is equal to the \emph{multivariate $G$-Tutte polynomial} \cite{LTY}. 
\end{itemize}

We denote by $E_{\underline{p}}$ the random subset obtained from $E$ by independently deleting each element $e \in E$ with probability $1 - p_e$. 
Let $f$ be a function on $\M$, that is, a function defined on $2^E$.
The probability that $E_{\underline{p}}$  is identical with  a subset $A\subseteq E$ is 
$ \underline{p}^A  \underline{(1-p)}^{E \setminus A}$. 
The expectation of $f(E_{\underline{p}})$ is given by
$$
\mathbb{E}\left[  f(E_{\underline{p}}) \right] =  \sum_{ A \subseteq E } f(A)  \underline{p}^A  \underline{(1-p)}^{E \setminus A}.
$$
We are interested in finding functions that have their expectations given by $\mathbf{Z}_{\M} (q, \underline{v})$.

\subsection{Background}
\label{subsec:Background} 
 An important contribution to the theory can be mentioned to the classical works of Welsh on the expectations of counting functions including the chromatic and flow polynomials on random graphs (e.g., \cite{Welsh96}). 
When the matroid is defined from an integral hyperplane arrangement, Ardila showed that the expectation of the characteristic polynomial of a random subarrangement can be computed by the finite field method \cite{Ardila07}. 
The methods used in the references above mostly apply to the case all $p_e$ have the same value, and give the expectations in terms of the classical Tutte polynomial. 
To deal with the general case, it is natural and essential to look for multivariate versions of the Tutte polynomial.
The expectations in some cases can be seen as specializations of (hence can be derived from) a convolution formula, which was hinted in the work of Kung \cite[Identity 1]{K10}. 
The Kung convolution formula proved for subset-corank polynomials (a close variation of the multivariate Tutte polynomial) of matroids applies to the general case, justifying the importance of multivariate generalizations for the Tutte polynomial.

The idea to associate a matroid with a multiplicity function is one of the new trends in recent years for decorating matroids (e.g, \cite{DM13, BM14, FM16, LTY, DFM18}). 
The classical Tutte polynomial of a matroid has been generalized accordingly in many different ways, and proved to have applications to various areas such as graph theory, arrangement theory, category theory, Ehrhart theory, etc.
However, less seems to be known about probabilistic aspects of the generalized Tutte polynomials, or expectation computing problems on random matroids with non-trivial multiplicity. 

\subsection{Results}
\label{subsec:Results} 
  Here we shall consider expectation computing problems on random \emph{ranked sets} with \emph{arbitrary} multiplicity, and the relationship between the ``multivariate Tutte polynomial" $\mathbf{Z}_{\M} (q, \underline{v})$ of $\M$ and the expected values of functions will be a main theme of this paper.
As in the case of matroids, one can produce ranked sets with multiplicity from $\M$ by performing the restriction and contraction operations  (Definition \ref{def:res-con-dua}).
Thus, we can construct the \emph{random restriction}  $\M|E_{\underline{p}}$ and the \emph{random contraction} $\M/E_{\underline{p}}$ of $\M$ from the random subset $E_{\underline{p}}$. 
We then give the computation on the expectations of several familiar and significant functions of these random variables.
Our main results are summarized in Tables \ref{tab:res} and \ref{tab:con}. 
In particular, if all $p_e$ are equal (the multivariate polynomials becomes bivariate), we obtain probabilistic interpretations of the arithmetic Tutte and $G$-Tutte polynomials. 
A simple, but seems to be new convolution-like formula for the Ehrhart polynomials of lattice zonotopes will also be given (Theorem \ref{for:convo-zono}).

  \begin{table}[htbp]
\centering
{\renewcommand\arraystretch{1.5} 
\begin{tabular}{K{6cm}|K{5.5cm}|K{2.5cm}}
Function & Expectation & Location    \\
\hline\hline
Multivariate Tutte polynomial of a rsm 
& $\mathbf{Z} _{\M} (t, \underline{pu})$ & Theorem \ref{for:before-chro}   \\
\hline
Rank-nullity polynomial of a rsm 
& $\mathbf{Z} _{\M} \left(\frac{y}{x}, \underline{py}\right)$ & Corollary \ref{for:rank}  \\
\hline
Tutte polynomial at $(2, y)$ of a rsm 
& $\mathbf{Z} _{\M} (y-1, \underline{p(y-1)})$ &  Corollary \ref{cor:2,y}  \\
\hline
Flow polynomial of a rsm 
& $\underline{(1-2p)}^{E} \mathbf{Z} _{\M} \left(t, \underline{\frac{tp}{1-2p}} \right)$ & Theorem \ref{thm:pri-flow} \\
\hline
Chromatic polynomial of a representable matroid with multiplicity
& $t^{ r(\Gamma) }\mathbf{Z}_{\M}(t, \underline{-p})$ & Theorem \ref{thm:primary-chro} \\
\hline
Euler characteristic of an abelian Lie group arrangement
& $\psi_G^{ r(\Gamma) }\mathbf{Z}_{\M}((-1)^{a+b}\psi_G, \underline{-p})$ & Remark \ref{rem:G-arr} \\
\hline
Number of layers of an intersection of hypersurfaces in an abelian Lie group arrangement
& $|F|^{r(\Gamma) }\underline{(1-p)}^{E} \mathbf{Z} _{\M} \left(|F|, \underline{\frac{p}{1-p}} \right)$ &  Theorem \ref{thm:cc}  \\
\hline
Br{\"a}nd{\'e}n-Moci generalization of the $q$-state Potts-model partition function of a representable matroid over $\Z$& $ |F|^{r(\Gamma)} \mathbf{Z}_{\M} (|F|,  \underline{pv})$ & Remark \ref{rem:Potts} \\
\hline
Br{\"a}nd{\'e}n-Moci multivariate Ehrhart polynomial of a zonotope 
& $ \mathbf{Z}_{\M} (q, \underline{qpv}) \mid_{q=0}$ & Theorem \ref{thm:ehr-strong}   \\
\hline
Number of  integer points in a half-open zonotope 
& $\underline{(1-p)}^{E} \mathbf{Z} _{\M} \left(1, \underline{\frac{kp}{1-p}} \right)$ &  Theorem \ref{thm:half-open} 
\end{tabular}
}
\bigskip
\caption{Expectations of functions of a random restriction of a ranked set with multiplicity (rsm) $\M$ given by the multivariate Tutte polynomial $\mathbf{Z} _{\M} (q, \underline{v})$.}
\label{tab:res}
\end{table}

  \begin{table}[htbp]
\centering
{\renewcommand\arraystretch{1.5} 
\begin{tabular}{K{5cm}|K{6.5cm}|K{2.5cm}}
Function & Expectation & Location    \\
\hline\hline
Tutte polynomial at $(x,2)$ of a rsm 
& $(x-1)^{r(E)}\underline{(1-p)}^{E} \mathbf{Z} _{\M} \left(x-1, \underline{\frac{1}{1-p}} \right) $ &  Corollary \ref{cor:2,y-dual}  \\
\hline
Characteristic polynomial of a rsm 
& $s^{r(E)}\underline{(1-p)}^{E} \mathbf{Z} _{\M} \left(s, \underline{\frac{2p-1}{1-p}} \right)$ & Theorem \ref{thm:chro-con} \\
\hline
Chromatic polynomial of a representable matroid with multiplicity
& $s^{ r(\Gamma) }\underline{(1-p)}^{E} \mathbf{Z} _{\M} \left(s, \underline{\frac{2p-1}{1-p}} \right)$ & Corollary \ref{cor:rep-chro-con}  
 \end{tabular}
}
\bigskip
\caption{Expectations of functions of a random contraction of a rsm  $\M$ given by the multivariate Tutte polynomial $\mathbf{Z} _{\M} (q, \underline{v})$.}
\label{tab:con}
\end{table}

\subsection{Methods}
\label{subsec:Methods} 
  We give a systematic study on a number of functions whose expectations are computable in terms of $\mathbf{Z}_{\M} (q, \underline{v})$ by dividing them into two models: \emph{monomials} on $\M$ defined by only the information of the ground set $E$, and \emph{polynomials} on $\M$ defined by the information of all subsets of $E$.

The polynomial model or we shall call it a convolution formula model, is based on a generalization of the Kung convolution formula, extending from matroids to ranked sets with multiplicities. However, we formulate it by means of the multivariate Tutte polynomials (Theorem \ref{thm:convol}) for the purpose of interpreting the expectations of polynomials in their accurate forms (Remark \ref{rem:prefer}).
This model enables us to compute the expectations of many familiar polynomials, such as
the chromatic, flow and Ehrhart polynomials, generalizing the Welsh's classical results on random graphs and Ardila's result on random arrangements. 

It turns out that some other invariants of $\M$ that are generally not ``good evaluations" of the polynomials mentioned above, such as the number of layers of an intersection of hypersurfaces in an abelian Lie group arrangement, and the number of lattice points in a half-open zonotope, still have the expectations given by $\mathbf{Z}_{\M} (q, \underline{v})$. 
The second model or monomial model, is designed to compute the expectations of such invariants.


\subsection{Organization of the paper}
\label{subsec:Organization} 
The remainder of the paper is organized as follows. 
In Section \ref{sec:Def-nota}, we give more details on operations on ranked sets with multiplicity and provide several examples. 
Inspired by matroid polynomials, we define various polynomials that can be associated with a ranked set with multiplicity.
In Section \ref{sec:convolution}, we give computation on the expectations through monomial and convolution formula models. 
We give meaning by adding geometric, enumerative or combinatorial flavors to many identities derived from the computation.
In Section \ref{sec:modi}, we mention some polynomial modifications and their expectations. 

\subsection{Acknowledgements}
\label{subsec:Acknowledgements} 
The author would like to thank Professor Masahiko Yoshinaga for many helpful suggestions and for allowing the author to include his result in Theorem \ref{thm:yos}.
The author is partially supported by JSPS Research Fellowship for Young Scientists Grant Number 19J12024.
 
\medskip

\section{Definitions and notations}
\label{sec:Def-nota} 

\subsection{Ranked sets with multiplicity}
\label{subsec:ranked-set} 
We first give more details on operations on ranked sets with multiplicity following \cite{ba-le}. 
\begin{definition} 
\label{def:rank-set} 
A \emph{ranked set} is a pair $(E,r)$, where $E$ is a finite set (the \emph{ground set}), and $r : 2^E \to \Z$ is a function (the \emph{rank function}).
A \emph{ranked set with multiplicity (rsm)} is a triple $\M=(E,r,m)$, where $(E,r)$ is a ranked set, and $m : 2^E \to R$ is a function (the \emph{multiplicity function}), where $R$ denotes a commutative ring with $1$. 
 \end{definition} 
 
The definition of ranked sets here is slightly more flexible than that in \cite{ba-le}: $r(\emptyset)$ can be non-zero. 
Two ranked sets with multiplicity $\M_1=(E_1, r_1, m_1)$ and $\M_2=(E_2, r_2, m_2)$ are said to be \emph{isomorphic}, written $\M_1 \simeq \M_2$,  if there is a bijection $\rho : E_1 \to E_2$ such that $r_2(\rho(A)) = r_1(A)$ and $m_2(\rho(A)) =m_1(A)$ for all $A \subseteq E$. 
In particular, if $E_1=E_2$, $r_1= r_2$, and $m_1=m_2$, we say that $\M_1$ and $\M_2$ are \emph{equal} and write $\M_1=\M_2$.
For $A \subseteq E$, denote $A^c:=E \setminus A$.
 
 \begin{definition} 
\label{def:res-con-dua} 
Let $\M=(E, r, m)$ be a rsm. Let $A \subseteq E$.
\begin{enumerate}[(1)]
\item   
 The \emph{restriction} $\M|A$ is the rsm $(A, r_{\M|A} , m_{\M|A})$, where $r_{\M|A}$ and $m_{\M|A}$ are the restrictions (as functions) of $r$ and $m$ to $A$. 
\item The \emph{contraction} $\M/A$ is the rsm $(A^c, r_{\M/A}, m_{\M/A})$, where $r_{\M/A}(B) := r(B \cup A) - r(A)$ and $m_{\M/A}(B) := m (B \cup A)$ for all $B\subseteq A^c$. Thus, $r_{\M/A}(\emptyset)=0$.
\item The \emph{dual} $\M^*$ is the rsm $(E, r^*, m^*)$, where $r^*(A) := |A| - r(E) + r(A^c)$ and $m^*(A) := m (A^c)$ for all $A\subseteq E$.  Thus, $r^*(\emptyset)=0$.
\end{enumerate}
 \end{definition} 
 
We shall need following lemma later. 
 \begin{lemma}
\label{lem:start}
Let $\M=(E, r, m)$ be a rsm. Then
\begin{enumerate}[(a)]
\item $(\M^*)^*=\M/\emptyset$. If particular,  $(\M^*)^*=\M/\emptyset=\M$ if $r(\emptyset)=0$.
\item $(\M/A)^* = \M^* | A^c$ for any $A\subseteq E$.   If particular,  $((\M^*)^*)^* = \M^*$.
\end{enumerate}
  \end{lemma}
 
 \begin{proof} The proof is not hard and it goes as follows: 
\begin{enumerate}[(a)]
\item For every $A\subseteq E$, $(r^*)^*(A)=|A|-r^*(E)+r^*(A^c)=r(A)-r(\emptyset)= r_{\M/\emptyset}(A)$. 
Moreover, $(m^*)^*(A)=m^*(A^c)=m(A)=m_{\M/\emptyset}(A)$.
\item For every $B\subseteq A^c$, $(r_{\M/A})^*(B)=|B|-r_{\M/A}(A^c)+r_{\M/A}(E \setminus (B \sqcup A))=|B|-r(E)+r(E \setminus B)= r^*(B)$. 
Moreover, $(m_{\M/A})^*(B)=m_{\M/A}(E \setminus (B \sqcup A))=m(E \setminus B)=m^*(B)$.
\end{enumerate}
\end{proof}

  Let us mention some typical examples of ranked sets with multiplicity that we frequently use in this paper. 
Throughout the paper, we use the word \emph{list} as a synonym of multiset.
  
 \begin{example}[Classical matroids]
\label{eg:matroid}
\quad
\begin{enumerate}[(1)]
\item   
A \emph{matroid} is a rsm, where $m$ is trivial, i.e., $m = 1$, and $r : 2^E \to \Z$ satisfies the following conditions: (i)  if $ A \subseteq E $, then $0\le r(A) \le |A| $, (ii) if $ A \subseteq  B \subseteq E $, then $r(A) \le r(B)$, (iii) if $A, B \subseteq E $, then $r(A  \cup  B )+r( A  \cap  B ) \le  r( A )+r( B )$. 
The dual, and every restriction/contraction of a matroid are matroids (e.g., \cite[\S 3 and \S 4]{O92}).
\item Let $\K$ be a field.
A matroid is said to be \emph{representable over $\K$}  (or \emph{$\K$-representable}) if it is isomorphic to a matroid  $(E,r)$, where $E$ is a finite list of vectors in a vector space $V$ over $\K$, and $r(A):=\dim\left( \langle A \rangle_{\K} \right)$ for $A  \subseteq E$, where $\langle A \rangle$ is the subspace of $V$ generated by $A$. 
The dual, and every restriction/contraction of a $\K$-representable matroid are $\K$-representable  (e.g., \cite[Corollary 2.2.9 and Proposition 3.2.4]{O92}).
\item The class of representable matroids includes graphic matroids (e.g., \cite[\S 1.1]{O92} and see also Remark \ref{rem:cha-chro}) and the matroids arising from central hyperplane arrangements (e.g., \cite[\S 3]{St07}, see also Example \ref{eg:g-mult}). We recall the latter. 
Let $\A$ be central hyperplane arrangement in a vector space $V$. 
Then $\A$ defines the matroid $\M(\A)=(\A, r)$, where $r(\scB):=\codim_V\left(\cap_{H \in\scB} H\right)$ for $\scB  \subseteq \A$. 
For each $H \in \A$, we choose a linear form $\alpha_H \in V^*=\Hom (V,\K)$ that defines
$H$, i.e., $H = \ker(\alpha_H)$. 
Thus the matroid $\M(\A)$ is represented by the list $\{ \alpha_H: H \in \A\} \subseteq V^*$. 
Conversely, every representable simple (no loops or multiple points) matroid $\M =(E,r)$ with $E \subseteq V$ defines an arrangement  $\A(\M) =\{ H_e : e \in E\}$ in $V^*$ given by $H_e=\{ \alpha \in V^* : \alpha(e)=0\}$.

\end{enumerate}
\end{example}

\begin{example}[$\Z$-representable matroids with multiplicity]
\label{eg:gale-dual}
Let $\Gamma\simeq \Z^s \oplus\Z/{d_1}\Z\oplus\cdots\oplus\Z/{d_n}\Z$ ($d_i>1$ for $1 \le i \le n$) be a finitely generated abelian group, and
let $E\subseteq\Gamma$ be a finite list of elements in $\Gamma$. 
For $A  \subseteq E$, define  $r(A) :=\rank(\langle A \rangle_{\Z})$ (rank of an abelian group), where $\langle A \rangle$ is the subgroup of $\Gamma$ generated by $A$. 
Set  $\M=(E,r)$.
By \cite[\S3.4]{DM13},  there exist two lists $Q, \widetilde{E} \subseteq \Z^{s+n}$ such that $Q \cap \widetilde{E}=\emptyset$, $r(Q) = n$, and $ \M \simeq \scN/Q$ (as matroids) under the isomorphism $\rho: E \to \widetilde{E}$, where $\scN=(Q \sqcup \widetilde{E},r)$. 
Moreover, for every $A  \subseteq E$, we have $\Gamma/\langle A\rangle \simeq \Z^{s+n}/\langle \rho(A) \sqcup Q\rangle$.

The ranked set $\M$ is called a \emph{representable matroid over $\Z$ (or a $\Z$-representable matroid)} in \cite[Definition 2.2]{FM16}.
Note that the matroid $\scN$ above is $\Q$-representable, thus by Example \ref{eg:matroid}, $\M$ is also $\Q$-representable in the usual sense. 
Thus a matroid is $\Q$-representable in the usual sense if and only if it is isomorphic to a $\Z$-representable matroid in the sense of Fink-Moci. We call a rsm a \emph{$\Z$-representable matroid with multiplicity} if the underlying ranked set is a $\Z$-representable matroid (i.e., it has the form $(E,r)$ above). 
 \end{example}

 \begin{example}[Arithmetic matroids]
\label{eg:arith-matroid}

\quad
\begin{enumerate}[(1)]
\item   
An \emph{arithmetic matroid} is a rsm, where $(E,r)$ is a matroid and the multiplicity function $m : 2^E \to \Z_{>0}$ satisfies the following conditions: (i) for all $A \subseteq E$ and $ a \in E $, if $r( A\cup \{ a\}) =r( A )$, then $m( A  \cup \{ a\})$ divides $m( A)$; otherwise $m(A )$ divides $m(A \cup \{ a\})$,
(ii)
if $[R, S]$ is a molecule with $S = R \sqcup F \sqcup  T$, then $
m( R ) \cdot m(S  ) = m(R  \sqcup  F  ) \cdot 
m(R \sqcup T )$, 
(iii) for every molecule $[R, S]$,
$
\rho(R,S):= (-1)^{|T|}\sum_{A\in[R, S]} (-1)^{|S|  - |A| }m(A) \ge 0$
 \cite{DM13, BM14}. 
 Here the set $[R,S]:=\{A  \subseteq E\mid R \subseteq A \subseteq S\}$ is called a \emph{molecule} if $S = R \sqcup F \sqcup  T$, and for each $A \in [R, S]$,
$r(A) = r(R) + |A \cap F |.$
Thus any matroid is an arithmetic matroid with trivial multiplicity. 
The dual, and every restriction/contraction of an arithmetic matroid are arithmetic matroids  \cite[\S 2.3]{DM13}.
\item An arithmetic matroid is said to be \emph{representable} if it is isomorphic to a rsm $\M=(E,r,m)$, where $(E,r)$ is a $\Z$-representable matroid (Example \ref{eg:gale-dual}), and the multiplicity is defined by $m(A):=|(\Gamma/\langle A\rangle)_{\tor}|$ (the \emph{arithmetic multiplicity}) for $A  \subseteq E$. 
Here  $(-)_{\tor}$ stands for the torsion subgroup. 
It follows from Example \ref{eg:gale-dual} that any representable arithmetic matroid is isomorphic to a contraction of another representable arithmetic matroid, which is represented by a finite list of elements in a free abelian group.

\item The dual, and every restriction/contraction of a representable arithmetic matroid are representable \cite[\S3.4 and Example 4.4]{DM13}. 
For example, if an arithmetic matroid $\M$ is represented by a list $E$ of elements in a finitely generated abelian group $\Gamma$, then the restriction $\M|A$ is represented by the list $A$ in $\Gamma$, and the contraction  $\M/A$ is represented by the list $\{\overline{e}: 
e\in E\setminus A\}$ of cosets in the group $\Gamma/\langle A\rangle$. 
\end{enumerate}
\end{example}

Before giving the final example of ranked sets (actually $\Z$-representable matroids) with (non-arithmetic) multiplicity, let us recall the notion of arrangement of subgroups over an abelian group following \cite[\S 3]{LTY}.
\begin{definition} 
\label{def:g-arrangement} 
Let $\Gamma$ be a finitely generated abelian group, and
let $E\subseteq\Gamma$ be a finite list of elements in $\Gamma$. 
Let $G$ be an arbitrary abelian group.
For each $e \in E$, we define the \emph{$G$-hyperplane} associated to $e$ as follows:
\begin{equation*}
H_{e, G}:=\{\varphi \in \Hom(\Gamma, G): \varphi (e)= 0_G\} \le \Hom(\Gamma, G). 
\end{equation*}
Then the \emph{$G$-arrangment} $ E (G)$ of $ E $ is the collection of the subgroups $H_{e, G}$
$$ E (G):=\{H_{e, G}: e\in E \}.$$ 
Assume further that $G$ is \emph{torsion-wise finite}, i.e., $G[d]:=\{x\in G : dx=0_G\}$ is finite for all $d\in\Z_{>0}$.
The \emph{$G$-multiplicity function} $m^G: 2^E \to \Z_{>0}$ is defined by 
\begin{equation*}
\label{eq:g-multiplicity}
m^G(A):=
|\Hom\left((\Gamma/\langle A \rangle)_{\tor}, G\right)|. 
\end{equation*}
Note that under the torsion-wise finiteness, $m^G(A)$ is a finite number for all $A\subseteq E$. 
\end{definition}  

From now on we assume that a group $G$ using to define $G$-multiplicity is always a torsion-wise finite abelian group. 
The main example we will use is abelian Lie group with finitely many connected components, or equivalently, the group of the form $G= F\times \R^a\times (\mathbb{S}^1)^b$, where $a,b\in \Z_{\ge0}$ and $F$ is a finite abelian group (to verify this equivalence see, e.g., \cite[Exercise 9.3.7]{HN12}). 
This setting includes several types of arrangements such as hyperplane arrangement (e.g.,  $G=\R$ or $\C$ \cite{OT92}) and toric arrangement (e.g., $G=\mathbb{S}^1$ or $\C^\times$ \cite{L12}) (we refer the reader to \cite[\S4.3]{LTY} for more details on the  specializations).

 \begin{example}[$\Z$-representable matroids with $G$-multiplicity]
\label{eg:g-mult}
Let $E$, $\Gamma$, $r$ be defined as in Example \ref{eg:gale-dual}, and let $m^G$ be defined as in Definition \ref{def:g-arrangement}. 
We call $(E, r, m^G)$ a \emph{$\Z$-representable matroid with $G$-multiplicity}. 
As quotients of $\Gamma$ are not necessarily torsion-free, the $G$-multiplicity is in general non-trivial (we may not want to start with  $(E, r)$ as a representable matroid over arbitrary field).
In particular, $m^{\mathbb{S}^1}$ is identical with the arithmetic multiplicity hence $(E, r, m^{\mathbb{S}^1})$  is a representable arithmetic matroid. 
Note that $m^{\{0\}} = 1$, thus $(E, r, m^{\{0\}})$  is a $\Z$-representable matroid. 
It should also be noted that $(E, r, m^G)$ is not an arithmetic matroid (and even neither a pseudo-arithmetic matroid nor a quasi-arithmetic matroid) for general group $G$ \cite[Remark 8.2 and Example 8.5]{LTY}. 
 \end{example}

\subsection{Tutte-related polynomials}
\label{subsec:Tutte-related} 
Throughout this subsection, $\M=(E, r, m)$ is a rsm. 
We recall the definitions of several typical polynomials that can be associated with $\M$.

 \begin{notation}
\label{def:multi}
We write $\underline{v}=\{v_e: e\in E\}$ for a labeled multiset
 of variables or numbers. 
For $A \subseteq E$, and a function $g(v)$ of $v$,  define 
$$\underline{g(v)}^A :=
\begin{dcases}
 \prod_{e \in A}g(v_e) \mbox{ if $A \ne \emptyset$},\\
1 \mbox{ if $A =\emptyset$}.
\end{dcases}
$$
In particular, when  $v_e=v$ for all $e\in E$, the above notation becomes
$$\underline{g(v)}^A  =g(v)^{|A|}.$$
In addition, if $\{u_e: e\in E\}$ is another multiset, then define
$$\underline{(uv)}^A := \prod_{e \in A}u_ev_e.$$
\end{notation}

\begin{definition} 
\label{def:r-multi}
The \emph{multivariate Tutte polynomial} $\mathbf{Z}_{\M} (q, \underline{v})$ of $\M$ is defined by 
\begin{equation*}
\mathbf{Z}_{\M} (q, \underline{v}):=
\sum_{A\subseteq E}m(A)q^{-r(A)}\underline{v}^A. 
\end{equation*}
Thus,
$
\mathbf{Z}_{\M^*} (q, \underline{v})=q^{r(E)-|E|} \underline{v}^E \mathbf{Z}_{\M} (q, \underline{qv^{-1}})
$ and 
$
\mathbf{Z}_{(\M^*)^*} (q, \underline{v})=q^{r(\emptyset)}  \mathbf{Z}_{\M} (q, \underline{v}).
$
In particular, if $v_e=v$ for all $e\in E$, we write  
$$Z _{\M} (q, v)=\sum_{ A \subseteq E}m(A)q^{-r(A)}v^{|A|}.$$
 \end{definition}

The name ``multivariate Tutte polynomial" originally refers to $\mathbf{Z}_{\M} (q, \underline{v})$ when $\M$ is a matroid \cite{sokal}.  
To avoid creating extra terminologies, we use the same name for $\mathbf{Z}_{\M} (q, \underline{v})$ when $\M$ is a rsm. 
A similar naming rule applies to other polynomials to come.

\begin{definition} 
\label{def:z-sc}
A variation of $\mathbf{Z}_{\M} (q, \underline{v})$, the \emph{subset-corank polynomial}  (e.g., defined in \cite{K10}) is defined by  
$$\mathbf{SC}_{\M} (q, \underline{v}):=q^{r(E)}\mathbf{Z}_{\M} (q, \underline{v}).$$
 \end{definition}

\begin{definition} 
\label{def:rank-nullity}
A variation of $Z _{\M} (q, v)$, the \emph{rank-nullity polynomial}   is defined by 
$$W(x,y):=\sum_{ A \subseteq E}m(A)x^{r(A)}y^{|A|-r(A)}.$$
In other words, $W(x,y)=Z _{\M} (y/x, y).$ 
This is also known as the rank polynomial, e.g., in \cite{Welsh96}. 
 \end{definition}

\begin{definition}
\label{def:r-Tutte}
The \emph{Tutte polynomial} $T _{\M} (x, y)$ of $\M$ is defined by 
$$T _{\M} (x, y)  :=\sum_{ A\subseteq E}m(A)(x-1)^{r(E)-r(A)}(y-1)^{|A|-r(A)}.
$$
Thus, $T _{\M^*} (y,x)= (y-1)^{r(\emptyset)} T _{\M} (x, y) =T _{(\M^*)^*} (x, y)$. 
In particular, $ T _{\M} (x, 2) =T _{\M^*} (2,x)$.
The polynomials $T _{\M} (x, y)$ and  $Z _{\M} (q, v)$ are  equivalent in the sense that one can be transformed to and from the other under the change of
variables:
\begin{align*}
T _{\M} (x, y) 
& =(x-1)^{r(E)}Z _{\M} ((x-1)(y-1), y-1), \\
Z _{\M} (q, v) & =\left( \frac{v}{q}\right)^{r(E)} T _{\M} \left( 1+\frac{q}{v}, 1+v\right).
\end{align*}
\end{definition}

  \begin{remark}
\label{rem:Strict}
We have considered $T _{\M} (x, y)$ as a polynomial in $R[(x-1)^{\pm1},(y-1)^{\pm1}]$. 
All computations and results in the paper valid for  $T _{\M} (x, y)$ are also valid for the \emph{nullity-corank polynomial} (usually known as the \emph{rank generating polynomial}) $T _{\M} (x+1, y+1)$, which is a Laurent polynomial in $R[x^{\pm1},y^{\pm1}]$. 
Most of our applications will be for matroids with multiplicity, in which case, $T _{\M} (x, y)$ is a polynomial in $R[x,y]$. 
\end{remark}

\begin{definition}\quad
\label{def:r-chro-flow}
\begin{enumerate}[(1)]
\item 
The \emph{flow polynomial} $F _{\M}(t)$ of $\M$ is defined by 
$$
F _{\M} (t)  := \sum_{ A\subseteq E}(-1)^{|E \setminus A|}m(A)t^{|A|-r(A)}.
$$
Thus, $F _{\M} (t)=(-1)^{|E|-r(E)}  T _{\M} (0,1-t)= (-1)^{|E|} Z _{\M} (t, -t)$ and $F_{(\M^*)^*}(t)=t^{r(\emptyset)} F _{\M} (t)$. 
\item  
The \emph{characteristic polynomial} $P _{\M}(t)$ of $\M$ is defined by 
$$
P _{\M} (t) 
 :=\sum_{ A\subseteq E}(-1)^{|A|}m(A) t^{r(E)-r(A)}.
$$
Thus, $P _{\M} (t)=(-1)^{r(E)} T _{\M} (1-t, 0)= t^{r(E)}Z _{\M} (t, -1)=F_{\M^*} (t)=P_{(\M^*)^*}(t)$.  
\end{enumerate}
\end{definition}

\begin{definition} \label{def:Q-chro}
Let $\M=(E,r,m)$ be a $\Z$-representable matroid with multiplicity (Example \ref{eg:gale-dual}).
Denote  $r(\Gamma):=\rank(\Gamma)$. 
The \emph{chromatic polynomial} $\chi _{\M}(t)$ of $\M$ is defined by 
$$
\chi _{\M}(t)  
 :=\sum_{ A\subseteq E}(-1)^{|A|}m(A) t^{ r(\Gamma) -r(A)}.
$$
Thus, $\chi _{\M}(t) = t^{ r(\Gamma) -r(E)}P _{\M} (t)=(-1)^{r(E)} t^{ r(\Gamma) -r(E)}T _{\M} (1-t, 0)= t^{ r(\Gamma) }Z _{\M} (t, -1)$. 
\end{definition}

  \begin{remark}
\label{rem:cha-chro}
In some contexts, especially in arrangement theory, the polynomial $\chi _{\M}(t)$ is usually known as the ``characteristic polynomial". 
For example, it coincides with the characteristic polynomial of the hyperplane/toric arrangement defined by $E$ under the suitable choice of multiplicity (e.g., \cite[Lemma 2.55]{OT92}, \cite[Theorem 5.6]{L12}, and see also  \cite[Corollary 3.8]{TY19} for more general result). 
We prefer calling $\chi _{\M}(t)$ the chromatic polynomial here because on the one hand, we want to distinguish between $\chi$ and the characteristic polynomial $P _{\M}(t)$; on the other hand, when $\M$ is the matroid defined from a graph, $\chi _{\M}(t)$ is exactly the classical chromatic polynomial of the graph. 
We briefly recall how it can be seen. 
Let $\mathcal{G} = (\scV,\scE)$ be a graph. Define a list of vectors $\scL = \{\alpha_e \mid e \in \scE\}$ in $\Z^{\scV}$ as follows. 
If $e=(ij) \in \scE$, let $\alpha_e$ be the vector with entry $j$ is $1$, entry $i$ is $-1$, and the other entries are $0$. 
Thus $\M=(\scL, r)$ is a $\Z$-representable matroid (or a $\Q$-representable matroid in the usual sense). 
(This construction is more or less a proof of the fact that graphic matroids are representable over every field.)
Moreover, $\chi _{\M}(t)$ is equal to the chromatic polynomial of $\mathcal{G}$ (e.g., \cite[Theorem 2.88]{OT92}).
\end{remark}

  \begin{remark}
\label{rem:another-rep}
One may start with a rsm $\M=(E,r,m)$, where $(E,r)$ is a representable matroid over a field $\K$ represented by a finite list of vectors in a vector space $V\simeq\K^\ell$, and define the corresponding chromatic polynomial as follows: 
$$
\chi' _{\M}(t)  
 :=\sum_{ A\subseteq E}(-1)^{|A|}m(A) t^{ \ell -r(A)}.
$$
The difference between $\chi$ and $\chi'$ is insignificant in this paper. 
Since $\chi'_{\M}(t) = t^{\ell}Z _{\M} (t, -1)$, all main results (Theorem \ref{thm:primary-chro} and Corollary \ref{cor:rep-chro-con}) available for $\chi$ are also available for $\chi'$ (up to a replacement of $r(\Gamma)$ by $\ell$). 
We prefer mentioning $\chi$ here because most important applications (Remark \ref{rem:G-arr}) will be for $\Z$-representable matroids with $G$-multiplicity. 
An application of $\chi'$ will also be mentioned in Formula \eqref{eq:B=B}.
\end{remark}


 \section{Expectations: monomial and convolution formula models}
\label{sec:convolution} 
Let $\mathbb{M}$ be the set of all isomorphism classes of ranked sets with multiplicity. 
Any function $H: \mathbb{M} \to R[x_1,\ldots,x_n]$ associates to a rsm $\M$ a polynomial $H_{\M}(x_1,\ldots,x_n)$ in $R[x_1,\ldots,x_n]$. 
For simplicity, we call $H_{\M}(x_1,\ldots,x_n)$ a polynomial associated with $\M$.

Now assume that $R=\R$.
Let $\M=(E, r, m)$ be a rsm. 
We denote by $E_{\underline{p}}$ the random subset obtained from $E$ by independently deleting each element $e \in E$ with probability $1 - p_e$ ($p_e \in [0, 1]$). 
If all $p_e$ have the same value, we simply write $E_p$. 
Thus we can construct the \emph{random restriction}   $\M|E_{\underline{p}}$ and the \emph{random contraction}  $\M/E_{\underline{p}}$ of $\M$. 
It is easy to see the probability that $E_{\underline{p}}$  is identical with  a subset $A\subseteq E$ is $\mathrm{Pr}(E_{\underline{p}}=A)= \underline{p}^A  \underline{(1-p)}^{E \setminus A}$. Let $\mathbb{E}[\mathscr{X}]$ denote the expectation of a random variable $\mathscr{X}$. 
We are interested in computing $\mathbb{E}[\mathscr{X}]$ when $\mathscr{X}$ is a function  of the random restriction/contraction. 
More precisely, if $H_{\M}(x_1,\ldots,x_n) \in \R[x_1,\ldots,x_n]$ is a polynomial associated with $\M$ (typically, any polynomial defined in Subsection \ref{subsec:Tutte-related}), we want to compute
\begin{equation}
\label{eq:exp-res}
\mathbb{E}\left[  H_{\M|E_{\underline{p}}} (x_1,\ldots,x_n)\right] =  \sum_{ A \subseteq E } H_{\M|A} (x_1,\ldots,x_n)  \underline{p}^A  \underline{(1-p)}^{E \setminus A},
\end{equation}
and $\mathbb{E}\left[  H_{\M/E_{\underline{p}}} (x_1,\ldots,x_n)\right]$ in terms of the multivariate Tutte polynomial $\mathbf{Z}_{\M} (q, \underline{v})$. 
In fact, we view $p_e$ $(e \in E)$ as general variables (and view $\R$ as a commutative ring $R$ with $1$) in all upcoming computations, and give them the values in $[0,1]$ only when the expectation takes effect.

\begin{remark}
\label{rem:applicativity}
Once the expectation of a function $\mathscr{X}$ is computed, the expectation of any function that is given (up to a factor independent of $E_{\underline{p}}$) by an evaluation of $\mathscr{X}$ follows immediately. 
We  shall sometimes use the phrase ``good evaluations" to indicate this.
 \end{remark}

We will use two models to compute the expectation, each corresponds to the case when the polynomial $H_{\M}$ in Formula \eqref{eq:exp-res} is a \emph{monomial} defined by only the information of the ground set $E$, or a \emph{polynomial} defined by the information of all subsets of $E$. 
Despite the ordering of models mentioned in Abstract and Introduction, we shall describe the monomial model first as it is simpler.  
 
It is also easy to see that $\mathrm{Pr}(E^c_{\underline{p}}=A)= \underline{(1-p)}^A  \underline{p}^{E \setminus A}$, where $E_{\underline{p}}^c=E\setminus E_{\underline{p}}$. 
The following duality of the expectation is often useful and will be used later in Theorem \ref{thm:chro-con} and Corollary \ref{cor:2,y-dual} (for the Tutte, characteristic and flow polynomials).

\begin{lemma}
\label{lem:duality}
Let $H_{\M}(x_1,\ldots,x_n) \in \R[x_1,\ldots,x_n]$ be a polynomial associated with  $\M$. Then
\begin{equation}
\label{eq:duality}
\mathbb{E}\left[  H_{\M|E_{\underline{p}}} (x_1,\ldots,x_n)\right] 
=
\mathbb{E}\left[  H_{\M|E^c_{\underline{1-p}}} (x_1,\ldots,x_n)\right].
\end{equation}
In addition, if $\sigma \in \mathfrak{S}_n$ is a permutation and $K: \mathbb{M} \to \R[x_1,\ldots,x_n]$ is a function such that $H_{\scN}(x_1,\ldots,x_n)=K_{\scN^*}(\sigma (x_1),\ldots,\sigma (x_n))$ for every rsm $\scN$, then
\begin{equation}
\label{eq:res-con}
\mathbb{E}\left[  H_{\M/E_{\underline{p}}} (x_1,\ldots,x_n)\right] 
=
\mathbb{E}\left[  K_{\M^*|E_{\underline{1-p}}} (\sigma (x_1),\ldots,\sigma (x_n))\right].
\end{equation}
 \end{lemma}
\begin{proof}
 Formula \eqref{eq:duality} is straightforward from the discussion above.  Formula \eqref{eq:res-con} follows from Formula \eqref{eq:duality} and Lemma \ref{lem:start}.
\end{proof}


 \subsection{Monomial model}
\label{subsec:simple} 
Let $\M=(E, r, m)$ be a rsm. 
The \emph{rank monomial} $X _{\M} (t)$ of $\M$ is defined by 
$$X _{\M} (t) := m(E)t^{-r(E)}.
$$
 \begin{theorem}
\label{thm:monomial}
\begin{equation}
\label{eq:monomial}
\underline{(1-v)}^{E} \mathbf{Z} _{\M} (t, \underline{v(1-v)^{-1}})
=
 \sum_{ A \subseteq E } X_{\M|A} (t)  \underline{v}^A  \underline{(1-v)}^{E \setminus A}.
\end{equation}
The expectation of the rank monomial of $\M|E_{\underline{p}}$ is given by
\begin{equation}
\label{eq:pe-monomial}
\mathbb{E}\left[X_{\M|E_{\underline{p}}}(t) \right] 
=
\underline{(1-p)}^{E} \mathbf{Z} _{\M} (t, \underline{p(1-p)^{-1}}).
\end{equation}
In particular,  if $p_e=p \in(0,1)$ for all $e\in E$, then
\begin{equation}
\label{eq:p-monomial}
\mathbb{E}\left[ X_{\M|E_p}(t) \right] 
=
p^{r(E)}(1-p)^{|E|-r(E)}t^{-r(E)}T_{\M}\left(1 + \frac{t(1-p)}p, \frac{1}{1 -p}\right).
\end{equation}
\end{theorem}
\begin{proof}
Formula \eqref{eq:monomial} follows directly from Definition \ref{def:r-multi}. 
The remaining formulas are straightforward. 
\end{proof}

Formula \eqref{eq:pe-monomial} has an interesting geometric-probabilistic interpretation. 
Let $G=(\mathbb{S}^1)^a\times \R^b\times F$, where $a,b\in \Z_{\ge0}$ and $F$ is a finite abelian group, and let $\M=(E, r, m^G)$ be a $\Z$-representable matroid with $G$-multiplicity (Example \ref{eg:g-mult}). 
We also recall the notion of the $G$-arrangement $E(G)$ in Definition \ref{def:g-arrangement}.  
For each $A \subseteq E$, denote $H_{A, G}:=\bigcap_{e\in A}H_{e, G}$, and write $\mathrm{cc}(H_{A, G})$ for the set of connected components (or layers) of $H_{A, G}$.  
By \cite[Proposition 3.6]{LTY}, we have
 \begin{equation}
 \label{eq:intersections}
\begin{aligned}
H_{A, G}
& \simeq  \Hom(\Gamma/\langle A \rangle, G) \\
&\simeq \Hom((\Gamma/\langle A \rangle)_{\tor}, G)\times F^{r(\Gamma) -r(A)}\times\left((\mathbb{S}^1)^a\times\R^b\right)^{r(\Gamma) -r(A)}.
\end{aligned} 
 \end{equation}
Each connected component of $H_{A, G}$ is isomorphic to $\left((\mathbb{S}^1)^a\times\R^b\right)^{r(\Gamma) -r(A)}$. 
Thus,
\begin{equation}
\label{eq:cc-comp}
|\mathrm{cc}(H_{A, G})| = m^G(A) \cdot |F|^{r(\Gamma) -r(A)}=|F|^{r(\Gamma)}X_{\M|A} (|F|).
\end{equation}

 \begin{theorem}
\label{thm:cc}
The expected number of the connected components of $H_{E_{\underline{p}},G}$ is given by
\begin{equation*}
\label{eq:cc}
\mathbb{E}\left[|\mathrm{cc}(H_{E_{\underline{p}}, G})| \right] 
=
|F|^{r(\Gamma)}\underline{(1-p)}^{E} \mathbf{Z} _{\M} (|F|, \underline{p(1-p)^{-1}}).
\end{equation*}
\end{theorem}
\begin{proof}
It follows  from Formulas \eqref{eq:pe-monomial} and \eqref{eq:cc-comp}.
\end{proof}

We immediately obtain the following result of Yoshinaga (unpublished).
 \begin{theorem}
\label{thm:yos}
Let $F$ be a finite abelian group. 
The expected number of homomorphisms from $\Gamma/\langle E_p \rangle$ to $F$ is given by
\begin{equation*}
\label{eq:cc}
\mathbb{E}\left[|\Hom(\Gamma/\langle E_p \rangle, F)|\right] 
=
|F|^{r(\Gamma) -r(E)}p^{r(E)}(1-p)^{|E|-r(E)} T_{\M}\left(1 + \frac{|F|(1-p)}p, \frac{1}{1 -p}\right).
\end{equation*}
\end{theorem}
\begin{proof}
Set $G=F$ (i.e., $a=b=0$) in Formula \eqref{eq:intersections} and apply Formula \eqref{eq:p-monomial}.
\end{proof}

The \emph{set monomial} $Y _{\M} (t)$ of $\M$ is defined by 
$$Y _{\M} (\underline{t}) := m(E)\underline{t}^{E}.
$$
It is not hard to prove the following.
 \begin{theorem}
\label{thm:set-monomial}
\begin{equation*}
\label{eq:set-monomial-convo}
\underline{(1-v)}^{E} \mathbf{Z} _{\M} (1, \underline{tv(1-v)^{-1}})
=
 \sum_{ A \subseteq E } Y_{\M|A} (t)  \underline{v}^A  \underline{(1-v)}^{E \setminus A}.
\end{equation*}
The expectation of the set monomial of $\M|E_{\underline{p}}$ is given by
\begin{equation}
\label{eq:set-monomial}
\mathbb{E}\left[Y_{\M|E_{\underline{p}}}(t) \right] 
=
\underline{(1-p)}^{E} \mathbf{Z} _{\M} (1, \underline{tp(1-p)^{-1}}).
\end{equation}
In particular,  if $p_e=p \in(0,1)$ for all $e\in E$, then
\begin{equation*}
\label{eq:pe-set-monomial}
\mathbb{E}\left[ Y_{\M|E_p}(t) \right] 
=
p^{r(E)}(1-p)^{|E|-r(E)}t^{r(E)}T_{\M}\left(1 + \frac{1-p}{tp}, 1+\frac{tp}{1 -p}\right).
\end{equation*}
\end{theorem}

Formula \eqref{eq:set-monomial} also has an interesting enumerative-probabilistic interpretation. 
Let $\M=(E, r, m)$ be a representable arithmetic matroid (Example \ref{eg:arith-matroid}), where $E$ is a finite list of \emph{independent} vectors in $\Z^n$ for some $n \ge 0$. 
The list $E$ defines a half-open zonotope 
$$\mathcal{Z}(E)^\diamond := \left\{ \sum_{e \in E} \lambda_ee : 0 \le \lambda_e < 1, \forall e \in E\right\}$$
in the real vector space spanned by $E$. 
For a labeled multiset $\underline{k}=\{k_e: e\in E\}\subseteq \Z$, denote $\underline{k}\cdot E :=\{k_ee: e\in E\}$. 
 \begin{theorem}
\label{thm:half-open}
The expected number of integer points in $\mathcal{Z}( \underline{k}\cdot E_{\underline{p}} )^\diamond $ ($\underline{k} \subseteq \Z_{> 0}$) is given by
\begin{equation*}
\label{eq:half-open}
\mathbb{E}\left[ | \mathcal{Z}( \underline{k}\cdot E_{\underline{p}})^\diamond \cap \Z^n|  \right] 
=
\underline{(1-p)}^{E} \mathbf{Z} _{\M} (1, \underline{kp(1-p)^{-1}}).
\end{equation*}
\end{theorem}
\begin{proof}
It follows  from Formula \eqref{eq:set-monomial}  and the fact that $| \mathcal{Z}( \underline{k}\cdot A)^\diamond \cap \Z^n| = m(A) \underline{k}^{A}$ for any $A \subseteq E$ (e.g., \cite[Lemma 9.8]{BR07} and \cite[Proposition 10.1]{BM14}).
\end{proof}

\begin{remark}
\label{rem:MST}
One may consider a model in which the target function $\mathscr{X}$ is an exponential function, e.g., it is related to the expected value of the length of minimal spanning tree of a graph (e.g., \cite{F85, FS05}).
 \end{remark}


 \subsection{Convolution formula model}
\label{subsec:convolution} 
Let $U$ be a commutative ring with $1$.
The set of all functions $f : 2^E \to U$ has a ring structure given by point-wise
multiplication and addition. 
Thus, to a ranked set $(E,r)$ with two possibly different multiplicities $m_1,m_2 : 2^E \to U$, we can associate the multiplicity $m_1m_2$. 
\begin{lemma}[Convolution formula model]
\label{lem:wang-conv}
Let $f,g : 2^E \to U$ be two functions. Then
\begin{equation}
\label{eq:wang-conv}
\sum_{T: \, T \subseteq E }  (-1)^{|T|} (fg)(T) 
=
 \sum_{A: \, A \subseteq E } 
\left[\sum_{B: \, B\subseteq A}  (-1)^{|B|} f(B)   \right]
 \left[ \sum_{T: \,  A\subseteq T \subseteq E } (-1)^{|T\setminus A|} g(T)\right].
\end{equation}
 \end{lemma}
\begin{proof}
 This formula is equivalent to  \cite[Theorem 1.1]{W15} when setting the poset be $(2^E,\subseteq)$. A direct proof is easy, and it goes as follows: the right hand side is equal to
\begin{equation}
\label{eq:cancel}
\sum_{B, T: \, B \subseteq T\subseteq E} (-1)^{|T|} f(B)g(T)  \left[ \sum_{A: \, B\subseteq A \subseteq T }
(-1)^{|A|-|B|} \right],
\end{equation}
 which is equal to the left hand side. 
 Note that the sum inside the bracket in Formula \eqref{eq:cancel} equals $0$ except in the case $B = T$, when it equals $1$.
\end{proof}
 
 We call Formula \eqref{eq:wang-conv} a ``convolution formula model" for computing expectation. 
Comparing with Formula \eqref{eq:exp-res}, the ``$B$" sum inside the first square bracket plays a role of the outcomes $H_{\M|A}$, and the ``$T$" sum in the second bracket plays a role of the probabilities $ \underline{p}^A  \underline{(1-p)}^{E \setminus A}$. 
The ``$B$" sum suggests that polynomials that we want to compute their expectations should be defined by the information of all subsets of the rsm's ground set. 
 The model is named by inspiration of the following generalization of a convolution formula  of Kung \cite[Identity 1]{K10}, extending from matroids to ranked sets with multiplicities.

\begin{theorem}
\label{thm:convol}
Let $(E,r)$ be a ranked set with two multiplicities $m_1, m_2$. Then 
\begin{equation}
\label{eq:convol}
\mathbf{Z} _{(E, r, m_1m_2)}(ts, \underline{uv})
=
 \sum_{ A \subseteq E }s^{-r(A)} \underline{(-v)}^A  \mathbf{Z}_{(E, r, m_1)|A}(t, \underline{-u})
\mathbf{Z}_{(E, r, m_2)/A}(s, \underline{v}).
\end{equation}
\end{theorem}
\begin{proof}
Set $f(S) = m_1(S)t^{-r(S)}\underline{u}^S$ and $g(S) = m_2(S)s^{-r(S)}\underline{(-v)}^S$ for every $S  \subseteq E$ in Formula  \eqref{eq:wang-conv}. 
Note that by definition, $\mathbf{Z}_{(E, r, m_2)/A}(s, \underline{v}) =  \sum_{T: \,  A\subseteq T \subseteq E } m_2(T)s^{-r(T)+r(A)}\underline{v}^{T\setminus A}$.
\end{proof}

\begin{remark}
\label{rem:differ} 
Convolution formulas for ranked sets with multiplicities were studied in \cite{ba-le}. 
The method we used here seems similar but a bit more direct and can be well-applied for multivariate polynomials. Also, we do not require a ``normalization" $r(\emptyset)=0$ in any rsm.
\end{remark}

\begin{remark}
\label{rem:prefer} 
 When $(E, r, m_1)$ and $(E, r, m_2)$ are classical matroids (i.e., $(E, r)$ is a matroid and $m_1=m_2 = 1$), Formula  \eqref{eq:convol} is equivalent to  \cite[Identity 1]{K10} which is formulated by means of subset-corank polynomials $ \mathbf{SC}_{\M} (q, \underline{v})$. 
In this case, although these two convolution formulas are equivalent, Formula  \eqref{eq:convol} is better suited to our purpose as we are mainly interested in computing the expectations of the polynomials in their accurate forms. 
For example, the Kung convolution formula does not produce the expectation $\mathbb{E}\left[  \mathbf{SC}_{\M|E_{\underline{p}}} (q, \underline{v}) \right]$, but instead give  
$$\mathbb{E}\left[  q^{-r(E_{\underline{p}})}\mathbf{SC}_{\M|E_{\underline{p}}} (q, \underline{v}) \right]= q^{-r(E)}\mathbf{SC}_{\M} (q, \underline{pv}).$$ 
This is the same as saying that 
$$\mathbb{E}\left[  \mathbf{Z}_{\M|E_{\underline{p}}} (q, \underline{v}) \right]=\mathbf{Z} _{\M} (q, \underline{pv}),$$
which we will show in Theorem \ref{for:before-chro} for any rsm $\M$. 
Moreover, it is also convenient to work with the formula above as our applications mainly are ``good evaluations" of $\mathbf{Z} _{\M} (q, \underline{v})$ (e.g., Theorem \ref{thm:ehr-strong} and Remark \ref{rem:Potts}).
The correction factor is actually a subtle obstacle and a more detailed clarification will come in Remark \ref{rem:model-limit}.
\end{remark}
Now let us mention one important consequence of Formula  \eqref{eq:convol}.

\begin{theorem}
\label{for:TutteDFM}
Let $(E,r)$ be a ranked set with two multiplicities $m_1, m_2$. Then 
\begin{equation*}
\label{eq:DFM}
T_{(E, r, m_1m_2)} (1-ab, 1-cd)
=
 \sum_{ A \subseteq E }a^{r(E)-r(A)}d^{|A|-r(A)} T_{(E, r, m_1)|A} (1-a, 1-c)
T_{(E, r, m_2)/A} (1-b, 1-d).
\end{equation*}
\end{theorem}
\begin{proof}
Set $t=ac$, $s=bd$, and $u_e=c$, $v_e=-d$ for all $e\in E$ in Formula  \eqref{eq:convol}. 
\end{proof}

\begin{remark}
\label{rem:r-convo-gen}
Theorem \ref{for:TutteDFM} has some notable specializations.  It corresponds to
 \begin{enumerate}[(a)]
 \item \cite[Theorem 4]{ba-le}, when $a=1$, $b=1-x$, $c=1-y$, $d=1$, 
\item (a simplification of) \cite[Theorem 10.9]{DFM18}, when $(E, r, m_1)$ and $(E, r, m_2)$ are arithmetic matroids, 
\item \cite[Theorem 8.6]{LTY}, when $(E, r, m_1)$ and $(E, r, m_2)$ are $\Z$-representable matroids with $G$-multiplicity (Example \ref{eg:g-mult}), and $a, b, c, d$ are given in (a), 
\item \cite[Identity (5.2)]{Welsh96}, when $(E, r, m_1)$ and $(E, r, m_2)$ are classical matroids, and $a=1-x$, $b=\frac{\theta+1}\theta$, $c=1-y$, $d=\frac{\theta}{\theta+1}$ (this specialization is rather less trivial).
\end{enumerate}
Neither of (b) and (c) is a specialization of the other (see the final comment in Example \ref{eg:g-mult}). 
The formula mentioned in (d) plays a crucial role in computing the expectations of the chromatic and flow polynomials of random subgraphs in \cite{Welsh96}.
We will extend these results to ranked sets with multiplicity (Theorems \ref{thm:primary-chro} and \ref{thm:pri-flow}). 
\end{remark}

With the convolution formula \eqref{eq:convol}, we are able to compute the expectations of several polynomials by appropriately specializing the variables. 
\begin{theorem}
\label{for:before-chro}
Let $\M=(E, r, m)$ be a rsm. Then
$$
\mathbf{Z} _{\M} (t, \underline{uv})
=
 \sum_{ A \subseteq E } \mathbf{Z}_{\M|A} (t, \underline{u})  \underline{v}^A  \underline{(1-v)}^{E \setminus A}.
$$
The expectation of the multivariate  Tutte polynomial  of $\M|E_{\underline{p}}$   is given by
\begin{equation}
\label{eq:before-welsh}
\mathbb{E}\left[  \mathbf{Z}_{\M|E_{\underline{p}}} (t, \underline{u}) \right] 
=
\mathbf{Z} _{\M} (t, \underline{pu}).
\end{equation}
\end{theorem}
\begin{proof}
To prove the first statement, set $m_1= m$, $m_2 = 1$, $s=1$, $u_e=-u_e$, $v_e=-v_e$ for all $e\in E$, and leave $t$ unchanged in Formula \eqref{eq:convol}. 
The second statement is straightforward. 
\end{proof}

Formula \eqref{eq:before-welsh} gives an interesting interpretation of  the multivariate arithmetic Tutte polynomial in connection with lattice point counting functions.
Let $\M=(E, r, m)$ be a representable arithmetic matroid where $E$ is a finite list of elements in a lattice $\Gamma\subseteq\R^n$ for some $n \ge 0$. 
The list $E$ defines a zonotope 
$$\mathcal{Z}(E) := \left\{ \sum_{e \in E} \lambda_ee : 0 \le \lambda_e \le 1, \forall e \in E\right\}$$
in the real vector space spanned by $E$. 
For a polytope $\mathcal{P}$ in $\R^n$ with the property that all vertices of the polytope are points of the lattice $\Gamma$,  the \emph{Ehrhart polynomial} of $\mathcal{P}$ with respect to $\Gamma$
is defined by 
$$\mathrm{Ehr}_\mathcal{P}(k)=\mathrm{Ehr}_\mathcal{P}(\Gamma; k):=|k\mathcal{P} \cap \Gamma|.$$ 
A multivariate version of the Ehrhart polynomial of $\mathcal{Z}(E)$ is defined in \cite[\S 10]{BM14} as follows. 
Recall that for $\underline{k}=\{k_e: e\in E\}\subseteq \Z$, $\underline{k}\cdot E$ denotes $\{k_ee: e\in E\}$. 
For $\underline{k} \subseteq \Z_{> 0}$, the Br{\"a}nd{\'e}n-Moci multivariate Ehrhart polynomial of $\mathcal{Z}(E)$ with respect to $\Gamma$ is defined by
\begin{equation*}
\label{eq:multi-BM}
\mathbf{Ehr}_E (\underline{k})=\mathbf{Ehr}_E (\Gamma; \underline{k}):=| \mathcal{Z}( \underline{k}\cdot E) \cap \Gamma|.
\end{equation*}
In particular, if $k_e=k$ for all $e\in E$, then $ \mathcal{Z}( k\cdot E)= k \mathcal{Z}(E)$.
Thus, 
$$
\mathbf{Ehr}_E (k) = \mathrm{Ehr}_{\mathcal{Z}( E)}( k).
$$
It is proved in  \cite[Proposition 10.1]{BM14} that    
\begin{equation}
\label{eq:multi-BM-thm} 
\mathbf{Ehr}_E (\underline{v}) = \mathbf{Z}_{\M} (q, \underline{qv}) \mid_{q=0} \,= \sum_{A:\, \text{$A$ is independent}}m(A)\underline{v}^A.
\end{equation}
When $v_e=k$  for all $e\in E$, the formula above specialzes to \cite[Theorem 3.2]{DM12} which asserts that  the Ehrhart polynomial of the zonotope $\mathcal{Z}(E)$ can be computed by the corresponding arithmetic Tutte polynomial,
\begin{equation}
\label{eq:BM-ori} 
 \mathrm{Ehr}_{\mathcal{Z}(E)}(k)  
=
k^{r(E)} T_{\M} ( 1+k^{-1}, 1).
\end{equation}

\begin{theorem}
\label{thm:ehr-strong}
The expectation of the Br{\"a}nd{\'e}n-Moci multivariate Ehrhart polynomial of $\mathcal{Z}(E_{\underline{p}})$ is given by
\begin{equation}
\label{eq:multi-ehr-strong}
\mathbb{E}\left[    \mathbf{Ehr}_{E_{\underline{p}}} (\underline{v})  \right] 
=
\mathbf{Ehr}_E (\underline{pv}).
\end{equation}
\end{theorem}
\begin{proof}
 It follows from Formulas \eqref{eq:before-welsh} and \eqref{eq:multi-BM-thm} that
 $$
 \mathbb{E}\left[    \mathbf{Ehr}_{E_{\underline{p}}} (\underline{v})  \right] 
=
\mathbb{E}\left[  \mathbf{Z}_{\M|E_{\underline{p}}} (q, \underline{qv})  \mid_{q=0} \right] 
= \mathbf{Z}_{\M} (q, \underline{qpv}) \mid_{q=0}
=
\mathbf{Ehr}_E (\underline{pv}).
$$
\end{proof}

\begin{theorem}
\label{thm:ehr}
The expectation of the Ehrhart polynomial of $\mathcal{Z}(E_{\underline{p}})$ is given by
\begin{equation}
\label{eq:multi-ehr}
\mathbb{E}\left[\mathrm{Ehr}_{\mathcal{Z}(E_{\underline{p}})}(t) \right] 
=
\mathbf{Ehr}_E (\underline{pt}).
\end{equation}
In particular,  if $p_e=p \in(0,1]$ for all $e\in E$, then
\begin{equation}
\label{eq:ehr}
\mathbb{E}\left[\mathrm{Ehr}_{\mathcal{Z}(E_p)}(t) \right] 
 = (pt)^{r(E)} T_{\M} ( 1+(pt)^{-1}, 1).
\end{equation}
\end{theorem}
\begin{proof}
Formula \eqref{eq:multi-ehr} follows from Formula \eqref{eq:multi-ehr-strong} by setting $v_e=t$  for all $e\in E$.
If $p_e=p \in(0,1]$ for all $e\in E$, then by Formula \eqref{eq:BM-ori} 
$$
\mathbb{E}\left[\mathrm{Ehr}_{\mathcal{Z}(E_p)}(t) \right] 
 =\mathbf{Ehr}_E (pt)  = \mathrm{Ehr}_{\mathcal{Z}( E)}(pt)
 = (pt)^{r(E)} T_{\M} ( 1+(pt)^{-1}, 1).
$$
\end{proof}

\begin{remark}
\label{rem:actually}
 In fact, Formula \eqref{eq:BM-ori} (\cite[Theorem 3.2]{DM12}) can be considered as a specialization of Formula \eqref{eq:ehr} obtained by setting $p=1$.
 \end{remark}
 
Let $\square_d:= [0,1]^d$ be the unit $d$-cube in $\R^d$. 
Note that every zonotope is a projection of the unit cube, and $\square_d$ itself is also a zonotope $\square_d=\mathcal{Z}(U_d)$ defined by the standard basis $U_d$ for $\R^d$. 
From Formula \eqref{eq:multi-ehr-strong}, we derive a convolution-like formula for the Br{\"a}nd{\'e}n-Moci multivariate Ehrhart polynomials. 
\begin{theorem}
\label{for:convo-zono}
We have
\begin{equation*}
\label{eq:convo-zono}
\mathbf{Ehr}_E (\Gamma; \underline{pv})
=
 \sum_{ A \subseteq E } \underline{p}^{A} \cdot \mathbf{Ehr}_A (\Gamma; \underline{v})
 \cdot \mathbf{Ehr}_{U_{|E\setminus A|}} (\Z^{E\setminus A}; \underline{-p}),
\end{equation*}
where the notation $\underline{-p}$ in the final term indicates $-p_e$ for $e \in E\setminus A$.
In particular,  if $p_e=p, v_e=t$ for all $e\in E$, then
$$
\mathrm{Ehr}_{\mathcal{Z}(E)} (\Gamma; pt)
=
 \sum_{ A \subseteq E } p^{|A|} \cdot \mathrm{Ehr}_{\mathcal{Z}(A)}(\Gamma; t) \cdot \mathrm{Ehr}_{\square_{|E\setminus A|}} (\Z^{E\setminus A}; -p).
 $$

\end{theorem}
\begin{proof}
If $ \underline{p} = \{p_i: 1 \le i \le d\} \subseteq \Z_{> 0}$, then $ \mathcal{Z}( \underline{p}\cdot U_d) = [0,p_1] \times \cdots  \times[0,p_d]  \subseteq \R^d$. Thus
$$\mathbf{Ehr}_{U_d} (\Z^{d}; \underline{p})= | \mathcal{Z}( \underline{p}\cdot U_d) \cap\Z^{d}|=\prod_{i=1}^d (1+p_i).$$
The rest follows from Formula \eqref{eq:multi-ehr-strong}.
 \end{proof}

We continue with some applications of Theorem \ref{for:before-chro} following the strategy in Remark \ref{rem:applicativity}.

\begin{corollary}
\label{for:rank}
The expectation of the rank-nullity polynomial of $\M|E_{\underline{p}}$   is given by
\begin{equation*}
\label{eq:rank}
\mathbb{E}\left[ W_{\M|E_{\underline{p}}} \left( x, y\right) \right] 
=
\mathbf{Z} _{\M} (y/x, \underline{py}).
\end{equation*}
\end{corollary}
\begin{proof}
 
It follows directly from Theorem \ref{for:before-chro} and Definition \ref{def:rank-nullity}. 
\end{proof}

\begin{remark}
\label{rem:Potts}
Let $F$ be a finite abelian group. 
Let $\M=(E, r, m^F)$ be a $\Z$-representable matroid with $F$-multiplicity (Example \ref{eg:g-mult}). 
Br{\"a}nd{\'e}n-Moci  \cite[\S 7]{BM14} defined the following polynomial as a generalization of the \emph{$q$-state Potts-model partition function}
$$\mathbf{Z}_E(\Gamma, F, \underline{v}) := \sum_{\phi \in \Hom(\Gamma, F)} \prod_{e \in E}(1 + v_e\delta(\phi(e),0)),$$
where $\delta$ is the Kronecker delta. By \cite[Example 4.15]{LTY},
$$\mathbf{Z} _E(\Gamma, F, \underline{v})  = |F|^{r(\Gamma)} \mathbf{Z}_{\M} (|F|,  \underline{v}).$$
Thus by Theorem \ref{for:before-chro},
\begin{equation*}
\label{eq:Potts}
\mathbb{E}\left[ \mathbf{Z} _{E_{\underline{p}}}(\Gamma, F, \underline{v})  \right] 
=
 |F|^{r(\Gamma)} \mathbf{Z}_{\M} (|F|,  \underline{pv}).
\end{equation*}
\end{remark}

\begin{theorem}
\label{thm:primary-chro}
The expectation of the chromatic polynomial of $\M|E_{\underline{p}}$ is given by
\begin{equation}
\label{eq:rep-region}
\mathbb{E}\left[ \chi_{\M|E_{\underline{p}}}(t) \right] 
=
t^{ r(\Gamma) }\mathbf{Z}_{\M}(t, \underline{-p}).
\end{equation}
In particular,  if $p_e=p \in(0,1]$ for all $e\in E$, then
\begin{equation}
\label{eq:primary-chro-confirm}
\mathbb{E}\left[  \chi_{\M|E_p}(t) \right] 
=
(-p)^{r(E)}t^{ r(\Gamma) -r(E)}T_{\M}\left(1 - \frac{t}p, 1 -p\right).
\end{equation}
\end{theorem}
\begin{proof}
It follows directly from Theorem \ref{for:before-chro} and Definition \ref{def:Q-chro}.
\end{proof}

\begin{remark}
\label{rem:chro}
When $\M$ is a matroid, Formula \eqref{eq:primary-chro-confirm} recovers \cite[Theorem 3.9]{Ardila07} for central $\Z$-arrangements and \cite[Theorem 5]{Welsh96} for graphs. 
With the same specialization, Corollary \ref{eq:rank} recovers \cite[Proposition 4]{Welsh96}.
\end{remark}

\begin{remark}
\label{rem:G-arr}
Let $G=(\mathbb{S}^1)^a\times \R^b\times F$, where $a,b\in \Z_{\ge0}$ and $F$ is a finite abelian group, and let $\M=(E, r, m^G)$ be a $\Z$-representable matroid with $G$-multiplicity. 
We also recall the notion of the $G$-arrangement $E(G)$ in Definition \ref{def:g-arrangement}.  
It is proved that the characteristic polynomial of the intersection poset of   layers of $E(G)$ (provided that $a+b>0$) \cite[Corollary 3.8]{TY19}, the Poincar\'e polynomial of the complement of $E(G)$ (provided that $b>0$) \cite[Theorem 7.8]{LTY}, and the Euler characteristic of $E(G)$ are ``good evaluations" 
of the chromatic polynomial (usually known as the $G$-characteristic polynomial) of $\M$. 
Thus the expectations of these functions can be computed by using Formula \eqref{eq:rep-region}. 
For example, by \cite[Theorem 5.2]{LTY}, the Euler characteristic $\psi_{E(G)}$ of $E(G)$ (for every $a, b \ge 0$) is given by 
$$\psi_{E(G)}=(-1)^{(a+b)r(\Gamma) }\chi_{\M}((-1)^{a+b}\psi_G).$$ 
In particular, the above formula  specializes to the known formulas on the number of regions of a real central hyperplane and toric arrangements when $G=\R$ and $G=\mathbb{S}^1$, respectively \cite{Z75, ERS09}. 
It follows from Formula \eqref{eq:rep-region} that
\begin{equation}
\label{eq:euler-region}
\mathbb{E}\left[  \psi_{E_{\underline{p}}(G)} \right] 
=
\psi_G^{ r(\Gamma) }\mathbf{Z}_{\M}((-1)^{a+b}\psi_G, \underline{-p}).
\end{equation}
\end{remark}

Now we want to work with the flow polynomial, however, Theorem \ref{for:before-chro} is no longer applicable. 
We need to choose a different specialization of Formula \eqref{eq:convol}. 
A slightly more general formula will come in Theorem \ref{for:multi-convo-res}.

\begin{theorem}
\label{thm:pri-flow}
Let $\M=(E, r, m)$ be a rsm. Then
\begin{equation}
\label{eq:primary-flow-confirm}
\underline{(1-2v)}^{E} \mathbf{Z} _{\M} (t, \underline{tv(1-2v)^{-1}})
=
 \sum_{ A \subseteq E } (-1)^{|A|}  \mathbf{Z}_{\M|A} (t,-t)  \underline{v}^A  \underline{(1-v)}^{E \setminus A}.
\end{equation}
Denote $B := \{ e \in E : p_e = \frac12\}$.
The expectation of the flow polynomial of $\M|E_p$  is given by
\begin{equation}
\label{eq:primary-flow-resolve}
\mathbb{E}\left[  F_{\M|E_{\underline{p}}} (t) \right] = 
\begin{dcases}
2^{-|E|} t^{|E|-r(E)} m(E) \mbox{ if $B = E$},\\ 
2^{-|B|} t^{|B|-r(B)}  \underline{(1-2p)}^{E \setminus B}\mathbf{Z} _{\M/B} (t, \underline{tp(1-2p)^{-1}}) \mbox{ if $B \subsetneq E$}.
\end{dcases}
\end{equation}
In particular,  if $B \subsetneq E$ and $p_e=p \in(0,1]\setminus\{\frac12\}$ for all $e\in E\setminus B$, then
$$
\mathbb{E}\left[  F_{\M|E_p} (t) \right] 
=
2^{-|B|} t^{|B|-r(B)} p^{r_{\M/B}}(1-2p)^{r^*_{\M/B}}T _{\M/B} \left(\frac{1-p}p, 1 +\frac{tp}{1-2p}\right),
$$ 
where, $r_{\M/B} := r(E)-r(B)$ and $r^*_{\M/B} := |E\setminus B|-r_{\M/B}$.
\end{theorem}

\begin{proof}
Setting  $m_1= m$, $m_2 = 1$, $s=1$,  $u_e=t$, $v_e=v_e(1-2v_e)^{-1}$ for all $e\in E$, and leaving $t$ unchanged in Formula \eqref{eq:convol}, we obtain Formula \eqref{eq:primary-flow-confirm}.
The right hand side of Formula \eqref{eq:primary-flow-confirm} is exactly $ \mathbb{E}\left[  F_{\M|E_{\underline{p}}} (t) \right]$ (after setting $v_e=p_e$ for all $e\in E$). Using the expansion of its left hand side, we obtain
\begin{equation}
\label{eq:primary-flow-expand}
 \mathbb{E}\left[  F_{\M|E_{\underline{p}}} (t) \right] = \sum_{A\subseteq E}m(A)t^{|A|-r(A)}\underline{p}^A   \underline{(1-2p)}^{E\setminus  A}.
\end{equation}
If $B=E$, then it is easy to see that 
$$ \mathbb{E}\left[  F_{\M|E_{\underline{p}}} (t) \right] =2^{-|E|} t^{|E|-r(E)} m(E).$$ 
Otherwise, we may write Formula \eqref{eq:primary-flow-expand} as 
$$ \mathbb{E}\left[  F_{\M|E_{\underline{p}}} (t) \right]  
=2^{-|B|} \sum_{A: \, B \subseteq A\subseteq E}m(A)t^{|A|-r(A)}\underline{p}^{A\setminus B}   \underline{(1-2p)}^{E\setminus  A}.$$
It is because if $B \nsubseteq A$, there exists $e \in B \setminus A \subseteq E \setminus A$, then $\underline{(1-2p)}^{E\setminus  A}=0$. Thus
  \begin{align*}
 \mathbb{E}\left[  F_{\M|E_{\underline{p}}} (t) \right]  
 & = 2^{-|B|} t^{|B|-r(B)}  \underline{(1-2p)}^{E \setminus B}  \sum_{A: \, B \subseteq A\subseteq E}m(A)t^{-(r(A)-r(B))}   \underline{(tp(1-2p)^{-1})}^{A\setminus  B} \\
 & = 2^{-|B|} t^{|B|-r(B)}  \underline{(1-2p)}^{E \setminus B}\mathbf{Z} _{\M/B} (t, \underline{tp(1-2p)^{-1}}).
\end{align*}

\end{proof}

\begin{remark}
\label{rem:flow-resolve}
Let $\M$ be a matroid in Theorem \ref{thm:pri-flow}, we recover  \cite[Theorem 6]{Welsh96}. 
\end{remark}

Now we do the computation on the random contraction.
\begin{proposition}
\label{prop:t=1}
 Let $\M=(E, r, m)$ be a rsm. Then
$$
\mathbf{Z} _{\M} (s, \underline{uv})
= 
 \sum_{ A \subseteq E } s^{-r(A)}  \underline{(-v)}^A  \underline{(1-u)}^{A}\mathbf{Z}_{\M/A} (s, \underline{v}).
$$
\end{proposition}
\begin{proof}
Set $m_1= 1$, $m_2 = m$,  $t=1$, and leave $s,  \underline{u},  \underline{v}$ unchanged in Formula \eqref{eq:convol}.
\end{proof}

\begin{corollary}
\label{cor:t=1,v=-1}
$$
s^{r(E)}\mathbf{Z} _{\M} (s, \underline{-u})
= 
 \sum_{ A \subseteq E } \underline{(1-u)}^{A}P_{\M/A} (s).
$$
\end{corollary}
\begin{proof}
Set $v_e=-1$ for all $e\in E$ in Proposition \ref{prop:t=1}.
\end{proof}
 
\begin{theorem}
\label{thm:chro-con}
\begin{equation}
\label{eq:primary-chro-con-confirm}
s^{r(E)}\underline{(1-u)}^{E} \mathbf{Z} _{\M} (s, \underline{(2u-1)(1-u)^{-1}})
=
 \sum_{ A \subseteq E }P_{\M/A} (s)  \underline{u}^A  \underline{(1-u)}^{E \setminus A}.
\end{equation}
Denote $C := \{ e \in E : p_e = 1\}$.
The expectation of the  characteristic polynomial of $\M/E_p$  is given by
\begin{equation}
\label{eq:primary-chro-con-cor}
\mathbb{E}\left[P_{\M/E_{\underline{p}}} (s) \right] 
=
\begin{dcases}
 m(E) \mbox{ if $C = E$},\\
s^{r(E)-r(C)}\underline{(1-p)}^{E\setminus C} \mathbf{Z} _{\M/C} (s, \underline{(2p-1)(1-p)^{-1}})\mbox{ if $C \subsetneq E$}.
\end{dcases}
\end{equation}
In particular,  if $C \subsetneq E$ and $p_e=p \in[0,1)\setminus\{\frac12\}$ for all $e\in E\setminus C$, then
$$
\mathbb{E}\left[P_{\M/E_p} (s) \right] 
= (2p-1)^{r_{\M/C}}(1-p)^{r^*_{\M/C}}T _{\M/C} \left(1 +\frac{s(1-p)}{2p-1}, \frac{p}{1-p}\right),
$$ 
where, $r_{\M/C} = r(E)-r(C)$ and $r^*_{\M/C} = |E\setminus C|-r_{\M/C}$.
\end{theorem}

\begin{proof}
We can prove Formula \eqref{eq:primary-chro-con-confirm} in two ways. The first way is to use Formula \eqref{eq:primary-flow-confirm} and  the duality \eqref{eq:res-con} in Lemma \ref{lem:duality}.
The second way is to use Corollary \ref{cor:t=1,v=-1}, in which we set  $u_e=(1-2u_e)(1-u_e)^{-1}$ for all $e\in E$.
Formula \eqref{eq:primary-chro-con-cor} follows from Formula \eqref{eq:primary-chro-con-confirm} with similar technique used in Proof of Theorem \ref{thm:pri-flow}.
\end{proof}

\begin{remark}
\label{rem:chro-kung}
The ways of choosing variables in Proposition \ref{prop:t=1}, Corollary \ref{cor:t=1,v=-1}, and the second proof of Formula \eqref{eq:primary-chro-con-confirm}
already appeared in \cite[Identities 4 and 5]{K10}. 
We recover these identities and the probabilistic interpretation in  \cite[Page 623]{K10} when taking $\M$ be a matroid. 
\end{remark}

\begin{corollary}
\label{cor:rep-chro-con}
The expectation of the chromatic polynomial of $\M/E_{\underline{p}}$ is given by
$$
\mathbb{E}\left[ \chi_{\M/E_{\underline{p}}}(s) \right] 
=
s^{r(\Gamma) }\underline{(1-p)}^{E} \mathbf{Z} _{\M} (s, \underline{(2p-1)(1-p)^{-1}}).
 $$
\end{corollary}
\begin{proof}
It follows directly from Formula \eqref{eq:primary-chro-con-confirm} and the fact that $\chi _{\M/A}(s) = s^{ r(\Gamma) -r(E)}P _{\M/A} (s)$ for every $A \subseteq E$.
\end{proof}

Let us mention an application of Corollary \ref{cor:t=1,v=-1} to hyperplane arrangements. 
First, fix a subset $B \subseteq E$, set $\M=\M/B$ in Corollary \ref{cor:t=1,v=-1}, and use the fact that (for every rsm $\M$) $(\M/B)/A = \M/(B \sqcup A)$ whenever $A \cap B =\emptyset$, to obtain
\begin{equation}
\label{eq:M=M/B}
t^{r(E)-r(B)}\mathbf{Z} _{\M/B} (t, \underline{-u})
= 
 \sum_{T: \, B \subseteq T \subseteq E } \underline{(1-u)}^{T \setminus B}P_{\M/T} (t).
\end{equation}

We refer the reader to \cite[\S 2]{OT92} for more details on terminology and notation in hyperplane arrangement theory. 
Let $\A$ be central hyperplane arrangement in a vector space $V\simeq \K^\ell$, and let $\M(\A)=(\A, r)$ be the matroid defined by $\A$ (Example \ref{eg:matroid}). 
Let $L(\A)$ be the intersection poset of $\A$.
For $X\in L(\A)$, write ${\A}^{X}$ for the restriction of ${\A}$ to $X$, and ${\A}_{X}$ for the localization of ${\A}$ on $X$. 
It is a standard fact that the set of flats of $\M(\A)$ is $\{{\A}_{X} : X\in L(\A)\}$. 
Moreover,  $\M(\A^X) \simeq \M(\A)/{\A}_{X}$ (as matroids) for every $ X\in L(\A)$ (e.g., \cite[\S 3.8]{Alex15}). 
Let $p_{\A}(t)$ denote the characteristic polynomial of $\A$ (e.g., \cite[Definition 2.52]{OT92}).
It is known (e.g., \cite[Lemma 2.55]{OT92}) that $p_{\A}(t)$ can be expressed in terms of the characteristic polynomial $P_{\M(\A)}(t)$ of $\M(\A)$ as follows:
$$p_{\A}(t) =t^{\dim_V(\cap \A)} P_{\M(\A)}(t),$$
where $\cap \A : = \cap_{H \in\A} H$. 
Note that in our notation, $p_{\A}(t)$ is essentially equal to $\chi'_{\M(\A)}(t)$ (Remark \ref{rem:another-rep}). 
Thus for every $ X\in L(\A)$,
$$p_{\A^X}(t) =t^{\dim_V(\cap \A)} P_{\M(\A)/{\A}_{X}}(t).$$
Fix a flat $\scB$ of $\M(\A)$. 
We can write $\scB = \A_Y$ for the unique $Y = \cap \scB \in L(\A)$. 
Note that if $\M$ is a matroid and $T$ is a flat of $\M$, then $P_{\M/T} (t)=0$. 
Set $\M=\M(\A)$ and $B=\scB$ in Formula \eqref{eq:M=M/B}, to obtain
\begin{equation}
\label{eq:B=B}
t^{\dim(Y)}\mathbf{Z} _{\M(\A^Y)} (t, \underline{-u})
= 
 \sum_{\substack{X\in L(\A) \\   X \subseteq Y}} \underline{(1-u)}^{\A_X \setminus \A_Y}p_{\A^X}(t).
\end{equation}

Note that if  $\M$ is a rsm, then $Z _{\M} (q, 0)=m(\emptyset)q^{-r(\emptyset)}$.
The following well-known formula (e.g., \cite[(3.2)]{OS83}) is a special case of Formula \eqref{eq:B=B}.
\begin{corollary}
\label{cor:OS83}
Let $\A$ be central hyperplane arrangement. For a fixed $Y \in L(\A)$, we have
$$
 \sum_{\substack{X\in L(\A) \\   X \subseteq Y}} p_{\A^X}(t)=
t^{\dim(Y)}. $$
\end{corollary}
\begin{proof}
Set $u_H=0$ for all $H\in \A  \setminus \A_Y$ in Formula \eqref{eq:B=B}.
\end{proof}

\begin{remark}
\label{rem:CW-BM} 
A special case of  Corollary \ref{cor:t=1,v=-1} when the rsm's multiplicity is non-trivial is already known.
For example, if we let $\M$ be a $\Z$-representable matroid with $\Z/q\Z$-multiplicity (Example \ref{eg:gale-dual}), and set $u_e=0$ for all $e\in E$ (after a small modification: replace $P$ and $r(E)$ by $\chi$ and $r(\Gamma)$, respectively), we recover \cite[Corollary 4.8]{Tan18}.
\end{remark}

\begin{remark}
\label{rem:model-limit}
Here we mention some polynomials that we are unable to compute their expectations (regarding the random restriction) with the convolution formula model: the subset-corank, characteristic and Tutte polynomials. 
These polynomials have a common property that each has a term ``$r(E)$" in the power of a variable. 
This term becomes ``$r(A)$" in the outcomes $H_{\M|A}$ in Formula \eqref{eq:exp-res}, and prevents the cancelation of the sum inside the bracket in Formula \eqref{eq:cancel}. 
It would be interesting to find compatible models for the polynomials mentioned above. 
To make the expectation computable within the scope of this paper, one possible way is to modify the polynomial. 
We will see the computation on some polynomial modifications in the next section.
\end{remark}


 \section{Expectations of some modified polynomials}
\label{sec:modi} 
In this section, we mention some polynomial modifications and their expectations. 
Working with the modified polynomials is sometimes helpful to derive the results on particular evaluations of the accurate polynomial (e.g., Corollary \ref{cor:2,y}). 
A typical way is to start from the expectation of (a modification of) the multivariate Tutte polynomial, then specialize to that of (a modification of) Tutte/characteristic/flow polynomial. 
First, starting from Theorem \ref{for:before-chro}, we obtain the following.

\begin{proposition}
\label{prop:Tutte}
The expectation of a modification of the  Tutte polynomial of $\M|E_{\underline{p}}$ is given by
\begin{equation*}
\label{eq:tutte-1}
\mathbb{E}\left[ (x-1)^{-r(E_{\underline{p}})} T_{\M|E_{\underline{p}}} \left( x, y\right) \right] 
=
\mathbf{Z} _{\M} ((x-1)(y-1), \underline{p(y-1)}).
\end{equation*}
In particular, if $p_e=p \in(0,1]$ for all $e\in E$, then
\begin{equation}
\label{eq:tutte-same}
\mathbb{E}\left[ (x-1)^{-r(E_p)} T_{\M|E_p} \left( x, y\right) \right] 
=
 \left(  \frac{p}{x-1} \right)^{r(E)} T _{\M} \left(1 + \frac{x-1}p, 1 +p(y-1)\right).
\end{equation}
\end{proposition}

\begin{corollary}
\label{cor:2,y}
$$
\mathbb{E}\left[ T_{\M|E_{\underline{p}}} \left(2, y\right) \right] 
=
\mathbf{Z} _{\M} (y-1, \underline{p(y-1)}).
$$
In particular, if $p_e=p \in(0,1]$ for all $e\in E$, then
\begin{equation*}
\label{eq:2,y}
\mathbb{E}\left[ T_{\M|E_p} \left( 2, y\right) \right] 
=
p^{r(E)} T _{\M} \left(1 + \frac{1}p, 1 +p(y-1)\right).
\end{equation*}
\end{corollary}

\begin{corollary}
\label{cor:2,y-dual}
$$
\mathbb{E}\left[ T_{\M/E_{\underline{p}}} \left(x,2 \right) \right] 
=
(x-1)^{r(E)}\underline{(1-p)}^{E} \mathbf{Z} _{\M} (x-1, \underline{(1-p)^{-1}}).
$$
In particular, if $p_e=p \in(0,1]$ for all $e\in E$, then
\begin{equation*}
\label{eq:2,y}
\mathbb{E}\left[ T_{\M/E_p} \left( x,2\right) \right] 
=
(1-p)^{|E|-r(E)} T _{\M} \left(1 +(1-p)(x-1) , \frac{2-p}{1-p}\right).
\end{equation*}
\end{corollary}

\begin{proof}
By Definition \ref{def:r-Tutte}, Corollary \ref{cor:2,y} and  the duality \eqref{eq:res-con} in Lemma \ref{lem:duality}, we have
$$\mathbb{E}\left[ T_{\M/E_{\underline{p}}} \left(x,2 \right) \right] 
=
\mathbb{E}\left[ T_{\M^*|E_{\underline{1-p}}} \left( 2, x\right) \right] 
=
\mathbf{Z} _{\M^*} (x-1, \underline{(1-p)(x-1)}).
$$
The rest follows from Definition \ref{def:r-multi}.
\end{proof}

\begin{remark}
\label{rem:2,y}
When $\M$ is a matroid, Formula \eqref{eq:tutte-same} recovers \cite[Identity (5.3)]{Welsh96}. 
Corollary \ref{cor:2,y} in turn gives probabilistic interpretations of the (multivariate) Tutte polynomials related to several invariants of $\M$: at $y=1$, the expected number of independent sets when $\M$ is a matroid \cite[Theorem 3]{Welsh96},  the expected number of integer points in a lattice zonotope when $\M$ is a representable arithmetic matroid (cf., Formula \eqref{eq:multi-ehr}); at $y=0$, the expected number of acyclic orientations when $\M$ is a graphic matroid  \cite[Proposition 5]{Welsh96}, the expected number of regions of a real central hyperplane arrangement when $\M$ is the matroid defined from the arrangement \cite{Z75} (this is exactly Formula \eqref{eq:euler-region} when $G=\R$).
 \end{remark}

\begin{corollary}
\label{for:chro}
The expectation of a modification of the  characteristic polynomial of $\M|E_{\underline{p}}$ is given by
$$
\mathbb{E}\left[t^{-r(E_{\underline{p}})}P_{\M|E_{\underline{p}}} (t) \right] 
=
 \mathbf{Z} _{\M} (t, \underline{-p}).
$$
In particular,  if $p_e=p \in(0,1]$ for all $e\in E$, then
\begin{equation*}
\label{eq:chro-confirm}
\mathbb{E}\left[ t^{-r(E_p)} P_{\M|E_p} (t) \right] 
=
(-p)^{r(E)}t^{-r(E)}T _{\M} \left(1 - \frac{t}p, 1 -p\right).
\end{equation*}
\end{corollary}

\begin{corollary}
\label{for:mflow}
The expectation of a modification of the   flow polynomial of $\M|E_{\underline{p}}$   is given by
$$
\mathbb{E}\left[  (-1)^{|E_{\underline{p}}|}F_{\M|E_{\underline{p}}} (t) \right] 
=
\mathbf{Z} _{\M} (t, \underline{-tp}).
$$
In particular,  if $p_e=p \in(0,1]$ for all $e\in E$, then
$$
\mathbb{E}\left[ (-1)^{|E_p|}F_{\M|E_p} (t) \right] 
=
(-p)^{r(E)} T _{\M} \left(1 - \frac{1}p, 1 -tp\right).
$$
\end{corollary}

\begin{theorem}
\label{for:multi-convo-res}
\begin{equation}
\label{eq:multi-convo-res}
\underline{(1-2v)}^{E} \mathbf{Z} _{\M} (t, \underline{uv(2v-1)^{-1}})
=
 \sum_{ A \subseteq E } (-1)^{|A|}  \mathbf{Z}_{\M|A} (t,\underline{u})  \underline{v}^A  \underline{(1-v)}^{E \setminus A}.
\end{equation}
The expectation of a modification of  the multivariate  Tutte polynomial  of $\M|E_{\underline{p}}$   is given by
\begin{equation*}
\label{eq:multi-res}
\mathbb{E}\left[  (-1)^{|E_{\underline{p}}|} \mathbf{Z}_{\M|E_{\underline{p}}} (t, \underline{u}) \right] 
=
\underline{(1-2p)}^{E} \mathbf{Z} _{\M} (t, \underline{up(2p-1)^{-1}}).
\end{equation*}
\end{theorem}
\begin{proof}
Setting  $m_1= m$, $m_2 = 1$, $s=1$,  $u_e=-u_e$, $v_e=v_e(1-2v_e)^{-1}$ for all $e\in E$, and leaving $t$ unchanged in Formula \eqref{eq:convol}, we obtain Formula \eqref{eq:multi-convo-res}. 
\end{proof}

Starting from Theorem \ref{for:multi-convo-res}, we obtain Theorem \ref{thm:pri-flow} and the following.

\begin{proposition}
\label{for:Tutte2}
$$
\mathbb{E}\left[ (-1)^{|E_{\underline{p}}|}(x-1)^{-r(E_{\underline{p}})} T_{\M|E_{\underline{p}}} \left( x, y\right) \right] 
=
\underline{(1-2p)}^{E} \mathbf{Z} _{\M} ((x-1)(y-1), \underline{(y-1)p(2p-1)^{-1}}).
$$
In particular, if $p_e=p \in(0,1]\setminus\{\frac12\}$ for all $e\in E$, then
\begin{equation*}
\label{eq:tutte-same2}
\mathbb{E}\left[  \frac{(-1)^{|E_ p|}T_{\M|E_p} \left( x, y\right)}{ (x-1)^{r(E_p)} } \right] 
=
   \frac{(1-2p)^{|E|-r(E)} (-p)^{r(E)}}{(x-1)^{r(E)}} T _{\M} \left(1 + \frac{(x-1)(2p-1)}p, 1 +\frac{p(y-1)}{2p-1}\right).
\end{equation*}
\end{proposition}

\begin{corollary}
\label{for:chro2}
 
$$
\mathbb{E}\left[  (-1)^{|E_{\underline{p}}|} t^{-r(E_{\underline{p}})}P_{\M|E_{\underline{p}}} (t) \right] 
=
\underline{(1-2p)}^{E} \mathbf{Z} _{\M} (t, \underline{p(1-2p)^{-1}}).
$$
In particular, if $p_e=p \in(0,1]\setminus\{\frac12\}$ for all $e\in E$, then
\begin{equation*}
\label{eq:chro-confirm2}
\mathbb{E}\left[(-1)^{|E_ p|} t^{-r(E_p)} P_{\M|E_p} (t) \right] 
=
   \frac{(1-2p)^{|E|-r(E)} p^{r(E)}}{t^{r(E)}} T _{\M} \left(1 + \frac{t(1-2p)}p, \frac{1-p}{1-2p} \right).
\end{equation*}
\end{corollary}

\begin{corollary}
\label{for:0,0}
Let $\M=(E, r, m)$ be a rsm, where $(E, r)$ is a matroid consisting of only coloops (i.e., $r(E)=|E|$ whence $T_{\M}  ( x, y)=T_{\M}  ( x, 0)$). Then
$$
\mathbb{E}\left[ T_{\M|E_{\underline{p}}} \left( 0, 0\right) \right] 
=
\mathbb{E}\left[  (-1)^{|E_{\underline{p}}|}  P_{\M|E_{\underline{p}}} (1) \right] 
=
\underline{(1-2p)}^{E} \mathbf{Z} _{\M} (1, \underline{p(1-2p)^{-1}}).
$$
In particular, if $p_e=p \in(0,1]\setminus\{\frac12\}$ for all $e\in E$, then
\begin{equation*}
\label{eq:0,0}
\mathbb{E}\left[  T_{\M|E_p} \left( 0, 0\right) \right] 
=
\mathbb{E}\left[  (-1)^{|E_p|}  P_{\M|E_p} (1) \right] 
=
 p^{|E|} T _{\M} \left( \frac{1-p}p, \frac{1-p}{1-2p}\right).
\end{equation*}
\end{corollary}

\begin{remark}
\label{rem:0,0}
If $\M$ is a matroid, then $T_{\M} \left( 0, 0\right)=0$ unless $E=\emptyset$ in which case, we agree that $T_{\M} \left( x, y\right)=1$. 
Thus for matroids, we can actually compute the expectation $\mathbb{E}\left[ T_{\M|E_{\underline{p}}} \left( 0, 0\right) \right] =\underline{(1-p)}^{E}$. 
However,  when $m$ is non-trivial, the constant term of $T_{\M}$ is totally non-trivial. 
For a $\Z$-representable matroid satisfying the condition in Corollary \ref{for:0,0} with special multiplicity function $m$, e.g., $G$-multiplicity, the coefficients of $T_{\M}$ are non-negative \cite[Proposition 8.4]{LTY}. 
\end{remark}

\begin{theorem}
\label{for:multi-convo-con}
\begin{equation}
\label{eq:multi-convo-con}
 \underline{(1-u)}^{E} \mathbf{Z} _{\M} (s, \underline{v(1-2u)(1-u)^{-1}})
=
 \sum_{ A \subseteq E } s^{-r(A)}  \underline{(-v)}^A \mathbf{Z}_{\M/A} (s, \underline{v})  \underline{u}^A  \underline{(1-u)}^{E \setminus A}.
\end{equation}
The expectation of a modification of  the multivariate  Tutte polynomial  of $\M/E_{\underline{p}}$   is given by
\begin{equation*}
\label{eq:multi-con}
\mathbb{E}\left[ s^{-r(E_{\underline{p}})}  \underline{(-v)}^{E_{\underline{p}}} \mathbf{Z}_{\M/E_{\underline{p}}} (s, \underline{v}) \right] 
=
 \underline{(1-p)}^{E} \mathbf{Z} _{\M} (s, \underline{v(1-2p)(1-p)^{-1}}).
\end{equation*}
\end{theorem}
\begin{proof}
Setting  $m_1= 1$, $m_2 = m$,  $t=1$, $u_e=(1-2u_e)(1-u_e)^{-1}$, $v_e=v_e$ for all $e\in E$, and leaving $s$ unchanged in Formula \eqref{eq:convol}, we obtain Formula \eqref{eq:multi-convo-con}. 
\end{proof}

Starting from Theorem  \ref{for:multi-convo-con}, we obtain Theorem \ref{thm:chro-con} and the following.

\begin{corollary}
\label{cor:Tutte2}
$$
\mathbb{E}\left[ \frac{(-1)^{|E_{\underline{p}}|} (y-1)^{|E_{\underline{p}}|-r(E_{\underline{p}})} T_{\M/E_{\underline{p}}} \left( x, y\right)}{(x-1)^{r(E_{\underline{p}})} } \right] 
=
\underline{(1-p)}^{E} \mathbf{Z} _{\M} \left((x-1)(y-1), \underline{\frac{(y-1)(1-2p)}{(1-p)}}\right).
$$
In particular, if $p_e=p \in[0,1)\setminus\{\frac12\}$ for all $e\in E$, then
\begin{align*}
 & \mathbb{E}\left[  \frac{(-1)^{|E_ p|} (y-1)^{|E_p|-r(E_p)} T_{\M/E_p} \left( x, y\right)}{ (x-1)^{r(E_p)} } \right] =
   \frac{(1-p)^{|E|-r(E)} (1-2p)^{r(E)}}{(x-1)^{r(E)}} \\
   & \times T _{\M} \left(1 + \frac{(x-1)(1-p)}{1-2p}, 1 +\frac{(y-1)(1-2p)}{1-p}\right).
\end{align*}
\end{corollary}

\begin{corollary}
\label{for:chro2}
 
$$
\mathbb{E}\left[  (-1)^{| E_{\underline{p}}|} t^{|E_{\underline{p}}|-r(E_{\underline{p}})}F_{\M/E_{\underline{p}}} (t) \right] 
=
 \underline{(p-1)}^{E} \mathbf{Z} _{\M} (t, \underline{-t(1-2p)(1-p)^{-1}}).
$$
In particular, if $p_e=p \in[0,1)\setminus\{\frac12\}$ for all $e\in E$, then
\begin{equation*}
\label{eq:chro-confirm2}
\mathbb{E}\left[(-1)^{|E_ p|} t^{|E_ p|-r(E_p)} F_{\M/E_p} (t) \right] 
=
 (p-1)^{|E|-r(E)} (1-2p)^{r(E)} T _{\M} \left(\frac{p}{2p-1},1+ \frac{t(1-2p)}{p-1} \right).
\end{equation*}
\end{corollary}

\bibliographystyle{alpha} 
\bibliography{references}

\end{document}